\documentclass[12pt,reqno]{amsart}
\usepackage{amscd,amssymb,amsmath,color}
  \usepackage[all]{xy}
\numberwithin{equation}{section}
\usepackage{amsfonts}
\usepackage{mathtools}
\usepackage{amsthm}
\usepackage{color}
\usepackage{xcolor}
\usepackage{hyperref}
\usepackage[pagewise]{lineno} 
\usepackage{braket}
\usepackage[all]{xy}

\usepackage{todonotes}

\usepackage[capposition=top]{floatrow}

\usepackage{float}
\usepackage{todonotes}
\usepackage{blkarray}
\usepackage{pgf,tikz}
\usepackage{mathrsfs}
\usetikzlibrary{arrows}
\definecolor{cqcqcq}{rgb}{0.7529411764705882,0.7529411764705882,0.7529411764705882}
\definecolor{qqzzff}{rgb}{0.,0.6,1.}
\definecolor{wwccqq}{rgb}{0.4,0.8,0.}
\definecolor{ffqqqq}{rgb}{1.,0.,0.}
\usepackage{tikz}
\usetikzlibrary{shapes.geometric, arrows}
\input{prepictex.tex}
\input{pictex.tex}
\input{postpictex.tex}

\usepackage{todonotes}
\usepackage{pgf,tikz}

\title[\tiny{On The Classification and Description of Quantum Lens Spaces as Graph algebras}]{On The Classification and Description of Quantum Lens Spaces as Graph algebras}

\author{Thomas Gotfredsen}
\address{Roskilde Business College, Bakkesvinget 67, 4000 Roskilde, Denmark}
\email{tgo@rhs.dk, gotfredsen\_thomas@hotmail.com}

\author{Sophie Emma Zegers}
\address{Faculty of Mathematics and Physics, Charles University, Sokolovská 49/83, 186 75 Praha 8, Czech Republic}
\email{zegers@karlin.mff.cuni.cz, sophieemmazegers@gmail.com}

\date\today
\makeatletter
\@namedef{subjclassname@2020}{%
  \textup{2020} Mathematics Subject Classification}


\thanks{The work was supported by  the DFF-Research Project 2 on `Automorphisms and invariants of operator algebras', Nr. 7014--00145B. The second author was supported by the Carlsberg Foundation through an Internationalisation Fellowship.
}
\begin{document}

\begin{abstract}
We investigate quantum lens spaces, $C(L_q^{2n+1}(r;\underline{m}))$, introduced by
Brzeziński-Szymański as graph $C^*$-algebras. We give a new description of $C(L_q^{2n+1}(r;\underline{m}))$ as graph $C^*$-algebras amending an error in the original paper by Brzeziński-Szymański.
Furthermore, for $n\leq 3$, we give a number-theoretic invariant, when all but one weight are coprime to the order of the acting group $r$. This builds upon the work of Eilers, Restorff, Ruiz and Sørensen.
\end{abstract}

\theoremstyle{plain}
\newtheorem{theorem}{Theorem}[section]
\newtheorem{corollary}[theorem]{Corollary}
\newtheorem{lemma}[theorem]{Lemma}
\newtheorem{proposition}[theorem]{Proposition}
\newtheorem{conjecture}[theorem]{Conjecture}
\newtheorem{commento}[theorem]{Comment}
\newtheorem{problem}[theorem]{Problem}
\newtheorem{remarks}[theorem]{Remarks}
\newtheorem{notation}[theorem]{Notation}

\theoremstyle{definition}
\newtheorem{example}[theorem]{Example}

\newtheorem{definition}[theorem]{Definition}

\newtheorem{remark}[theorem]{Remark}

\newcommand{\Nb}{{\mathbb{N}}}
\newcommand{\Rb}{{\mathbb{R}}}
\newcommand{\Tb}{{\mathbb{T}}}
\newcommand{\Zb}{{\mathbb{Z}}}
\newcommand{\Cb}{{\mathbb{C}}}

\newcommand{\Af}{\mathfrak A}
\newcommand{\Bf}{\mathfrak B}
\newcommand{\Ef}{\mathfrak E}
\newcommand{\Gf}{\mathfrak G}
\newcommand{\Hf}{\mathfrak H}
\newcommand{\Kf}{\mathfrak K}
\newcommand{\Lf}{\mathfrak L}
\newcommand{\Mf}{\mathfrak M}
\newcommand{\Rf}{\mathfrak R}

\newcommand{\x}{\mathfrak x}
\def\C{\mathbb C}
\def\N{\mathbb N}
\def\R{\mathbb R}
\def\T{\mathbb T}
\def\Z{\mathbb Z}

\def\A{{\mathcal A}}
\def\B{{\mathcal B}}
\def\D{{\mathcal D}}
\def\E{{\mathcal E}}
\def\F{{\mathcal F}}
\def\G{{\mathcal G}}
\def\H{{\mathcal H}}
\def\J{{\mathcal J}}
\def\K{{\mathcal K}}
\def\LL{{\mathcal L}}
\def\N{{\mathcal N}}
\def\M{{\mathcal M}}
\def\N{{\mathcal N}}
\def\OO{{\mathcal O}}
\def\P{{\mathcal P}}
\def\Q{{\mathcal Q}}
\def\SS{{\mathcal S}}
\def\T{{\mathcal T}}
\def\U{{\mathcal U}}
\def\W{{\mathcal W}}

\def\ext{\operatorname{Ext}}
\def\span{\operatorname{span}}
\def\clsp{\overline{\operatorname{span}}}
\def\Ad{\operatorname{Ad}}
\def\ad{\operatorname{Ad}}
\def\tr{\operatorname{tr}}
\def\id{\operatorname{id}}
\def\en{\operatorname{End}}
\def\aut{\operatorname{Aut}}
\def\out{\operatorname{Out}}
\def\coker{\operatorname{coker}}
\def\pol{\mathcal{O}}

\newcommand{\SL}{{\text{SL}}}

\def\la{\langle}
\def\ra{\rangle}
\def\rh{\rightharpoonup}
\def\cl{\textcolor{blue}{$\clubsuit$}}

\def\bst{\textcolor{blue}{$\bigstar$}}

\newcommand{\norm}[1]{\left\lVert#1\right\rVert}
\newcommand{\inpro}[2]{\left\langle#1,#2\right\rangle}
\newcommand{\Mod}[1]{\ (\mathrm{mod}\ #1)}
\renewcommand{\thefootnote}{\alph{footnote}}
\renewcommand{\bibname}{References}

\maketitle
\section{Introduction}
In the study of noncommutative geometry many classical spaces have been given a quantum analogue. Due to Gelfand duality there exists an equivalence between the categories of commutative $C^*$-algebras and locally compact Hausdorff spaces. Hence, when studying quantum analogues of classical spaces one often thinks of them as algebras of continuous functions on a non-existing virtual space. 

A well-studied example is the quantum sphere by Vaksman and Soibelman \cite{vs}, from which we define quantum lens spaces as fixed point algebras under the action of finite cyclic groups. In noncommutative geometry quantum lens spaces are objects of increasing interest, see e.g. \cite{bf, abl, dl} where noncommutative line bundles with quantum lens spaces as total spaces are investigated. 

In \cite{hs} Hong and Szymański gave a description of quantum lens spaces as graph $C^*$-algebras. This description was extended in \cite{BS} by Brzeziński and Szymański to also include weights which are not necessarily coprime with the order of the acting finite cyclic group. Unfortunately the general description is incorrect, which was recently pointed out by Efren Ruiz. 

In the present paper we first describe a new graph which is a modified version of the one given by Brzeziński and Szymański and prove that quantum lens spaces are indeed graph $C^*$-algebras. Then we deal with classification of quantum lens spaces of dimension at most 7, with certain conditions on their weights. We remark that the work on classification has already been presented in an earlier preprint, unpublished, available on arXiv, \cite{GM21}, by the present authors. After the submission of the preprint to arXiv it was pointed out the authors by Efren Ruiz that the graph $C^*$-algebraic description of quantum lens spaces is incorrect in some cases. This affects to some extent the classification results presented in the first preprint. The present paper serves as an extension of the previous one, containing both the modified graph $C^*$-algebraic description and the adjusted classification results.

For the determination of isomorphism of quantum lens spaces, it is not sufficient only considering their K-groups and the order, \cite[Remark 7.10]{errs}. In \cite{errs} Eilers, Restorff, Ruiz and Sørensen came with an important classification result of finite graph $C^*$-algebras using the reduced filtered K-theory. As opposed to the classification of Cuntz-Krieger algebras given by Restorff in \cite{r}, which the result in \cite{errs} is based on, quantum lens spaces fall within the scope of this classification. As an application of the classification result, Eilers, Restorff, Ruiz and Sørensen investigated 7-dimensional quantum lens spaces for which all the weights are coprime with the order of the acting cyclic group $\Zb_r$. They managed to reduce the classification result to elementary matrix algebras using $SL_{\mathcal{P}}$-equivalence and to prove that the lowest dimension for which we get different quantum lens spaces is dimension $7$. Here they showed that there exist two different quantum lens spaces when $r$ is a multiple of 3, and precisely one when this is not the case. 

Further investigation of quantum lens spaces, as defined in \cite{hs}, was conducted in \cite{jkr} by Jensen, Klausen and Rasmussen using $SL_{\mathcal{P}}$-equivalence. For a fixed $r$ they showed how large the dimension of the quantum lens space $C(L_q^{2n+1}(r;m_0,...m_{n}))$ must be to obtain non-isomorphic quantum lens spaces. The work is based on computer experiments by Eilers, who came up with a suggestion for a number $s$ such that for $n<s$ the quantum lens spaces are all isomorphic.

In this paper we will extend the result by Eilers, Restorff, Ruiz and Sørensen to quantum lens spaces of dimension less than or equal to $7$ for which $\gcd(m_i,r)\neq 1$ for one and only one $i$. The work builds on computer experiments, which were made in collaboration with Søren Eilers. We use a program written by Eilers in Maple 2019\footnote{Maplesoft, a division of Waterloo Maple Inc., Waterloo, Ontario.}, which has been optimised slightly by the present authors. Concretely, the program computes the adjacency matrices and isomorphism classes given the order $r$ and the set $\lbrace \gcd(m_i,r)\colon i=0,\dots 3 \rbrace$. By considering various combinations of the values of $r$ and the weights, we came up with a suggestion for an invariant, depending on which weight that is not coprime with the order of the acting group. In this way, experiments have played a crucial role in determining the statement of the presented theorems.

The structure of the paper is as follows: In section \ref{section:preliminaries} we present the classification result by Eilers, Restorff, Ruiz and Sørensen in the case of type I graph $C^*$-algebras. Section
\ref{section:QLgraphs} contains a counterexample proving that the description of quantum lens spaces as graph $C^*$-algebras by Brzeziński and Szymański is incorrect. Moreover it contains a new proof that quantum lens spaces are indeed graph $C^*$-algebras. We emphasise that the content in this section has already been presented in the same format in \cite[Chapter 2]{M21}. In section \ref{section:classificationQLS} we describe the classification result by Eilers, Restorff, Ruiz and Sørensen in the setting of quantum lens spaces and present the classification for quantum lens spaces for which all the weights are copime to the order of the acting group. 

The procedure to classify quantum lens spaces follows by first constructing the adjacency matrices, which are presented in section \ref{adjacencymatrix}. Afterwards we calculate an invariant using $SL_{\mathcal{P}}$-equivalence which involves some long calculations. Therefore the proofs are postponed to section \ref{invariant}, and the main theorems (Theorem \ref{Thm:Main} \& \ref{5dimensional}) are stated in section \ref{classification}.
\\ 
\\
\textbf{Acknowledgements} The authors would like to thank Søren Eilers for helpful discussions as well as providing a program which has been crucial to the investigation. We would, furthermore, like to thank Efren Ruiz for pointing out the problem in the description of quantum lens spaces as graph $C^*$-algebras and for his important guidance and help in constructing a new description. The authors also gratefully acknowledge helpful comments and suggestions from Wojciech Szymański and James Gabe. We also thank the anonymous referee for comments and suggestions for improvements.

\section{Preliminaries}\label{section:preliminaries}
We recall first some concepts of graph $C^*$-algebras which are needed in this paper. A directed graph $E=(E^0,E^1,r,s)$ consists of a countable set $E^0$ of \textit{vertices}, a countable set $E^1$ of \textit{edges} and two maps $r,s: E^1\to E^0$ called the \textit{range map} and the \textit{source map} respectively. For an edge $e\in E^1$ from $v$ to $w$ we have $s(e)=v$ and $r(e)=w$. 
For a directed graph graph $E$, we let $A_E=[A(v,w)]_{v,w\in E^0}$ where $A(v,w)$ is the number of edges with source $v$ and range $w$. $A_E$ is called the adjacency matrix for $E$. Moreover we let $B_E:=A_E-I$.

A graph is called \textit{finite} if it has finitely many edges and vertices. For a directed graph $E$, we recall that a vertex is \textit{regular} if $s^{-1}(v)=\{e\in E^1| \ s(e)=v\}$ is finite and nonempty, it is called \textit{singular} if this is not the case. In the case were we have too many edges between two vertices to make a good drawing, we only draw one edge and indicate the number of edges as follows: $\bullet \overset{(m)}{\longrightarrow} \bullet$ if the number of edges is $m$. 

A \textit{path} $\alpha$ in a graph is a finite sequence $\alpha=e_1e_2\cdots e_n$ of edges satisfying $r(e_i)=s(e_{i+1})$ for $i=1,...,n-1$. A path $\alpha$ is called a \textit{cycle} if $s(\alpha)=r(\alpha)$ and a loop if $\alpha$ is a cycle of length one. It is called a \textit{return path} if $\alpha$ is a cycle and $r(e_i)\neq r(\alpha)$ for $i<n$. 

Let $v,w\in E^0$, if there is a path from $v$ to $w$ in the graph then we write $v\geq w$. A subset $H\in E^0$ is called \textit{hereditary} if $v\in H$ and $w\in E^0$ is such that $v\geq w$ then $w\in H$. 
\\

The graph $C^*$-algebra of a directed graph is defined as follows (see e.g. \cite{bpis,flr}). 
\begin{definition}\label{graphalgebra}
Let $E=(E^0,E^1,r,s)$ be a directed graph. The graph $C^*$-algebra $C^*(E)$
is the universal $C^*$-algebra generated by families of orthogonal projections $\{p_v | \ v\in E^0\}$ and partial isometries $\{s_e | \ e\in E^1\}$ with mutually orthogonal ranges (i.e. $s_e^*s_f=0, e\neq f$) 
subject to the relations 
\begin{itemize}
\item[] (CK1) $s_e^*s_e=p_{r(e)}$
\item[] (CK2) $s_es_e^*\leq p_{s(e)}$
\item[] (CK3) $p_v=\underset{s(e)=v}{\sum}s_e s_e^*$, if $\{e\in E^1| \ s(e)=v\}$ is finite and nonempty. 
\end{itemize}
\end{definition}
For at path $\alpha=e_1e_2\cdots e_n$ we let $s_\alpha=s_{e_1}s_{e_2}\cdots s_{e_n}$.

We can by universality define a circle action, called the \textit{gauge action}, $\gamma: U(1)\to \text{Aut}(C^*(E))$ for which 
$\gamma_z(p_v)=p_v \ \text{and} \ \gamma_z(s_e)=zs_e$
for all $v\in E^0, e\in E^1$ and $z\in U(1)$.

If the graph has finitely many vertices, we say that a nonempty subset $S\subseteq E^0$ is \textit{strongly connected} if for any pair of vertices $v,w\in S$ there exists a path from $v$ to $w$. It is called a \textit{strongly connected component} if it is a maximal strongly connected subset. We let $\Gamma_E$ be the set of all strongly connected components and all singletons of singular vertices which are not the base of a cycle. Moreover, a strongly connected component is called a \textit{cyclic component} if one of its vertices has exactly one return path. 

Let Prime$_\gamma(C^*(E))$ be the set of all proper ideals of $C^*(E)$ which are prime and gauge invariant. For at finite graph $E$ it follows by \cite[Lemma 3.16]{errs} that there exists a homeomorphism $\nu_E: \Gamma_E\to$ Prime$_\gamma(C^*(E))$  such that $\gamma_1\geq \gamma_2$ if and only if  $\nu_E(\gamma_1)\supseteq \nu_E(\gamma_2)$ for $\gamma_1,\gamma_2\in \Gamma_E$.

The structure of Prime$_\gamma(C^*(E))$, and hence of $\Gamma_E$, will become crucial in the application of $SL_{\mathcal{P}}$-equivalence. 

\subsection{Classification of graph $C^*$-algebras over finite graphs}
In this section we briefly describe the classification result of finite graphs by Eilers, Restorff, Ruiz and Sørensen in \cite{errs}. We restrict to the case of type I/postliminal $C^*$-algebras. By \cite[Lemma 4.20]{errs} and \cite{DHS03}, a graph $C^*$-algebra $C^*(E)$ is of type I if and only if no vertices support two distinct return paths. 

In \cite[Theorem 6.1]{errs} finite graphs are classified up to stable isomorphism by their \textit{ordered reduced filtered K-theory}, in which the main idea is to consider the K-theory of specific ideals, the corresponding quotients and the maps between them. We will not describe this further. Instead, we consider type I graph $C^*$-algebras for which the classification result is reduced to a question of $\SL_{\mathcal{P}}$-equivalence as presented in \cite{errs,errs2}. $\SL_{\mathcal{P}}$-equivalence boils down to elementary matrix algebras which makes it a very useful tool in applications. 
\\

We describe the definition of $\SL_{\mathcal{P}}$-equivalence, for this we let $\boldsymbol{n}=(n_i)_{i=1}^N,\boldsymbol{m}=(m_i)_{i=1}^N\in \Nb^N$ be multiindices and $|\boldsymbol{n}|=n_1+\cdots+n_N$. We denote by $\boldsymbol{1}$ the multiindex with $1$ on every entry. 
\begin{definition}
Let $\mathcal{P}=\{1,2,3,...N\}$ with $N\in\Nb$ be a partially ordered set with order denoted $\preceq$.
Let $\boldsymbol{m},\boldsymbol{n}\in \Nb^N$ be multiindices such that $|\boldsymbol{m}|>0$ and $|\boldsymbol{n}|>0$. Then $\mathfrak{M}_{\mathcal{P}}(\boldsymbol{m}\times\boldsymbol{n},\Z)$ is the set of block matrices
$$
B=\begin{pmatrix}
B\{1,1\} & \cdots & B\{1,N\} \\
\vdots & & \vdots \\
B\{N,1\} & \cdots & B\{N,N\}
\end{pmatrix}
$$
for which 
\begin{equation}\label{blockmatrix}
B\{i,j\}\neq 0 \Rightarrow i\preceq j,
\end{equation}
where $B\{i,j\}\in M(m_i\times n_i,\Zb)$. If $m_i=n_i=0$ then $B\{i,j\}$ is the empty matrix. Moreover, we denote $B\{i,i\}$ by $B\{i\}$. Note that condition \eqref{blockmatrix} implies that the matrices in $\mathfrak{M}_{\mathcal{P}}(\boldsymbol{m}\times\boldsymbol{n},\Z)$ are upper triangular block matrices. 
\end{definition}

Let $\mathfrak{M}_{\mathcal{P}}(\boldsymbol{n},\Z)$ denote $\mathfrak{M}_{\mathcal{P}}(\boldsymbol{n}\times\boldsymbol{n},\Z)$. 
We define $\SL_{\mathcal{P}}(\boldsymbol{n},\Z)$ to be the matrices in $\mathfrak{M}_{\mathcal{P}}(\boldsymbol{n},\Z)$ such that all the non-empty diagonal blocks have determinant $ 1$. 

\begin{definition} Let $A,B\in \mathfrak{M}_{\mathcal{P}}(\boldsymbol{m}\times\boldsymbol{n},\Z)$, we say that $A$ and $B$ are $\SL_{\mathcal{P}}$-equivalent if there exist $U\in \SL_{\mathcal{P}}(\boldsymbol{m},\Z)$ and $V\in \SL_{\mathcal{P}}(\boldsymbol{n},\Zb)$ such that $UAV=B$. 
\end{definition}

The block structure of $B_E$ for a finite graph $E$ is given by the conditions in \cite[Definition 4.15]{errs}. Here, a partial ordered set $\mathcal{P}$ is defined such that there is an order reversing isomorphism from $\mathcal{P}$ to $\Gamma_E$ and hence encodes the ideal structure. 

Let $E$ be a finite graph which has no vertices supporting two distinct return paths. From \cite[Definition 4.15]{errs} and the following remark it follows that $B_E$ can be assumed to have a $1\times 1$ block structure i.e. $B_E\in \mathfrak{M}^\circ_{\mathcal{P}}(\boldsymbol{1},\Z)$ (see \cite[Definition 4.15]{errs}). $SL_{\mathcal{P}}(\mathbf{1},\Z)$ is in this case given as the set of upper triangular matrices, $A=(a_{ij})$, with $1$ on the diagonal and which satisfies: 
$
a_{ij}\neq 0 \Rightarrow i\preceq j.
$

$SL_{\mathcal{P}}$-equivalence simplifies in this case, since the block structure consists of $1\times 1$ matrices. Hence working with $SL_{\mathcal{P}}$-equivalence becomes a linear problem. Note that $SL_{\mathcal{P}}(\mathbf{1},\Z)$ is a group under matrix multiplication. 
\\
Let ${E}_{\curlywedge}$ be the graph which is obtained from $E$ by adding a loop to all sinks in $E$. For two finite graphs $E$ and $F$ we say that $(B_E,B_F)$ is in standard form if the adjacency matrices for $E$ and $F$ have the same size and block structure, moreover they must also have the same temperatures i.e. the same types of gauge simple subquotients, see \cite[Definition. 4.22]{errs} for a precise definition. 
The partial ordered set $\mathcal{P}$ is defined such that there is an order reversing isomorphism from $\mathcal{P}$ to $\Gamma_E$ and $\Gamma_F$. 

Type I graph $C^*$-algebras are classified by the following result: 
\begin{theorem}[{\cite[Theorem 7.1]{errs}\cite[Proposition 14.8]{errs2}}]\label{Theorem:Classification}
Let $E$ and $F$ be finite graphs which have no vertices supporting two distinct return paths. If $(B_E,B_F)$ is in standard form, then $C^*(E)$ and $C^*(F)$ are isomorphic if and only if there exist matrices $U,V\in SL_{\mathcal{P}}(\mathbf{1},\Z)$ such that $U{B_E}_{\curlywedge}V={B_F}_{\curlywedge}$. 
\end{theorem}

\section{Quantum Lens spaces as graph $C^*$-algebras}\label{section:QLgraphs}
The quantum $(2n+1)$-sphere by Vaksman and Soibelman, denoted $C(S_q^{2n+1})$, is the universal $C^*$-algebra generated by $z_0,z_1,...,z_n$ with the following relations: 
\[
\begin{aligned}
&z_jz_i=qz_iz_j, \;\;\; \text{for} \; i<j, \ \ z_iz_j^*=qz_j^*z_i, \;\;\; \text{for} \; i\neq j, \\
&z_i^*z_i=z_iz_i^*+(1-q^2)\sum_{j=i+1}^n z_jz_j^*, \; \text{for} \ i=0,...,n, \\ 
&\sum_{j=0}^n z_jz_j^*=1,
\end{aligned} 
\]
where $q\in (0,1)$, see \cite{vs}. It was shown in \cite{HS02} that $C(S_q^{2n+1})\cong C^*(L_{2n+1})$ where the graph $L_{2n+1}$ has vertices $v_i, i=0,...,n$ with edges $e_{ij}, 0\leq i\leq j\leq n$ and $s(e_{ij})=v_i, r(e_{ij})=v_j$. 

Let $\underline{m}=(m_0,m_1,...,m_n)$ be a sequence of positive integers. The $C^*$-algebra $C(S_q^{2n+1})$ admits, by universality, an action of $\Z_r$ for any $r\in\mathbb{N}$, given by 
\[
\varrho_{\underline{m}}^r: z_i\mapsto \theta^{m_i}z_i,
\]
where $\theta$ is a generator of $\Z_r$. The quantum lens space $C(L_q^{2n+1}(r;\underline{m}))$ is defined as the fixed point algebra of $C(S_q^{2n+1})$ under this action.

The action $\varrho_{\underline{m}}^r$ on $C(S_q^{2n+1})$ translates under the isomorphism with $C^*(L_{2n+1})$, to the following action:
$$
s_{e_{ij}}\mapsto \theta^{m_i}s_{e_{ij}}, \ \ p_{v_i}\mapsto p_{v_i}. 
$$
We also denote this action by $\varrho_{\underline{m}}^r$. It then follows by \cite[Theorem 4.6]{C08} that 
\begin{equation}\label{corner}
C(L_q^{2n+1}(r;\underline{m}))\cong C(L_{2n+1})^{\varrho_{\underline{m}}^r}\cong \left(\sum_{i=0}^n p_{(v_i,0)}\right)C^*(L_{2n+1}\times_{
c}\mathbb{Z}_r)\left(\sum_{i=0}^n p_{(v_i,0)}\right). 
\end{equation}

The graph $L_{2n+1}\times_{
c}\mathbb{Z}_r$, called the \textit{skew product graph} labelled by $c: e_{ij}\to m_i \Mod{r}$, has vertices $(v_i,k), i=0,...,n, k=0,...,r-1$ and edges $(e_{ij}, k), i,j=0,...,n, i\leq j, k=0,...,r-1$. The source and range maps are given as follows: 
$$
s((e_{ij}, k))=(v_i,k-m_i \Mod{r}), \ \ r((e_{ij},k))=(v_j,k). 
$$

In \cite[Theorem 2.2]{BS}, it is stated that $C(L_q^{2n+1}(r;\underline{m}))$ is isomorphic to the graph $C^*$-algebra $C^*(L_{2n+1}^{r;\underline{m}})$. To define the graph $L_{2n+1}^{r;\underline{m}}$ we need the notion of an admissible path. 

\begin{definition}[\cite{BS}]
A path from $(v_i,s)$ to $(v_j,t)$ in $L_{2n+1}\times_{c}\Z_r$ is called \textit{admissible} if it does not pass through any $(v_\ell,k)$ for which $\ell=i,...,j$ and $k=0,...,\gcd(m_\ell,r)-1$. 
\end{definition}

\begin{remark}
Comparing with the notion of 0-simple paths from \cite[Definition 7.4]{errs}, it is clear that the 0-simple paths are exactly the admissible paths when all weights are coprime to the order of the acting group.
\end{remark}
\begin{definition}[\cite{BS}]
The graph $L_{2n+1}^{r;\underline{m}}$ has vertices $v_i^b,i=0,...,n, b=0,...,\gcd(m_i,r)-1$ and edges $e_{ij;a}^{st}$, $a=1,...,n_{ij}^{st}$ where 
$$n_{ij}^{st}=\text{the number of admissible paths from $(v_i,s$) to $(v_j,t)$}.$$

The source and range maps are given by
$$
s(e_{ij;a}^{st})=v_i^s, \ \ r(e_{ij;a}^{st})=v_j^t.
$$
\end{definition}
The following example, which was pointed out by Efren Ruiz, shows that it is not in general true that $C(L_q^{2n+1}(r;\underline{m}))$ is isomorphic to the graph $C^*$-algebra $C^*(L_{2n+1}^{r;\underline{m}})$ as stated in \cite[Theorem 2.2]{BS}. 

\begin{example}[Counterexample of {\cite[Theorem 2.2]{BS}}]\label{counterexample} Let $n=1, r=4$ and $\underline{m}=(2,1)$ then the skew product graph consists of two levels which both consists of four vertices, the first level consist of two cycles as follows:
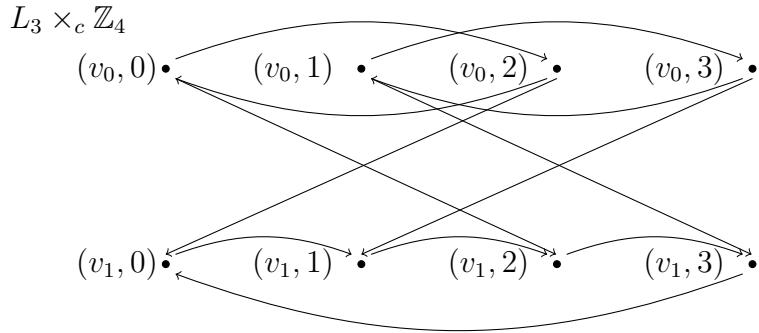
\begin{center}
\begin{tikzpicture}[scale=1.3]
\filldraw [black] (0,0) circle (1pt);
\filldraw [black] (2,0) circle (1pt);
\filldraw [black] (4,0) circle (1pt);
\filldraw [black] (6,0) circle (1pt);

\draw[->] (3.9,-0.1) to [out=-160,in=-20] (0.1,-0.1);
\draw[->] (0.1,0.1) to [out=20,in=160] (3.9,0.1);
\draw[->] (5.9,-0.1) to [out=-160,in=-20] (2.1,-0.1);
\draw[->] (2.1,0.1) to [out=20,in=160] (5.9,0.1);

\node at (-1,0.5)  {$L_3\times_c\Z_4$};
\node at (-0.5, 0)  {$(v_0,0)$};
\node at (1.3, 0)  {$(v_0,1)$};
\node at (3.3, 0)  {$(v_0,2)$};
\node at (5.3, 0)  {$(v_0,3)$};

\filldraw [black] (0,-2) circle (1pt);
\filldraw [black] (2,-2) circle (1pt);
\filldraw [black] (4,-2) circle (1pt);
\filldraw [black] (6,-2) circle (1pt);

\draw[->] (0.1,-1.9) to [out=20,in=160] (1.9,-1.9);
\draw[->] (2.1,-1.9) to [out=20,in=160] (3.9,-1.9);
\draw[->] (4.1,-1.9) to [out=20,in=160] (5.9,-1.9);
\draw[->] (5.9,-2.1) to [out=-160,in=-20] (0.1,-2.1);

\draw[->] (0.1,-0.1) to (3.99,-1.9);
\draw[->] (2.1,-0.1) to (5.99,-1.9);
\draw [->] (4,-0.1) to (0,-1.9);
\draw[->] (6,-0.1) to (2,-1.9);

\node at (-0.5, -2)  {$(v_1,0)$};
\node at (1.3, -2)  {$(v_1,1)$};
\node at (3.3, -2)  {$(v_1,2)$};
\node at (5.3, -2)  {$(v_1,3)$};

\end{tikzpicture}
\end{center}
We have 
$$
\begin{aligned}
C(L_q^3(4;(2,1)))&\cong (p_{(v_0,0)}+p_{(v_1,0)})C^*(L_3\times_c\Z_4)(p_{(v_0,0)}+p_{(v_1,0)}) \\
&= \overline{\span}\{s_{\mu}s_{\nu}^*| \ r(\mu)=r(\nu), s(\mu),s(\nu)\in \{(v_0,0),(v_1,0)\}\} \\
&=(p_{(v_0,0)}+p_{(v_1,0)})C^*(G)(p_{(v_0,0)}+p_{(v_1,0)})
\end{aligned}
$$
where $G$ is the subgraph of $L_3\times_c\Z_4$ for which 
$$
\begin{aligned}
G^0=\{(v_0,i),(v_1,j)| \ i=0,2, j=0,1,2,3\}, \ \ G^1=s^{-1}(E^0)\cap r^{-1}(E^0). 
\end{aligned}
$$
The range and source maps are the ones from $L_3\times_c\Z_4$ restricted to $G$. The graph $G$ is defined as above, since we have no paths from $(v_0,0)$ or $(v_1,0)$ to $(v_0,1)$ and $(v_0,3)$. Hence, we can remove the vertices $(v_0,i), i=1,3$ and their outgoing edges. Note that $G^0$ is the smallest hereditary subset of $(L_3\times_c \Zb_4)^0$ which contains $(v_0,0)$ and $(v_1,0)$. 

Since $G$ is a Cuntz-Krieger algebra it follows by \cite[Corollary 4.10]{AE15} that the corner is indeed a Cuntz-Krieger algebra, hence a graph $C^*$-algebra. 

The projection $p_{(v_0,0)}+p_{(v_1,0)}$ is full in $C^*(G)$. Indeed, let $I$ be the ideal generated by  $p_{(v_0,0)}+p_{(v_1,0)}$. By the Cuntz-Krieger relations we have 
$$
p_{(v_0,2)}=s_{(e_{00},2)}^*s_{(e_{00},2)}, \ \ s_{(e_{00},2)}s_{(e_{00},2)}^*\leq p_{(v_0,0)}.
$$
Hence 
$$s_{(e_{00},2)}^*=s_{(e_{00},2)}^*s_{(e_{00},2)}s_{(e_{00},2)}^*=s_{(e_{00},2)}^*s_{(e_{00},2)}s_{(e_{00},2)}^*p_{v_0,0}\in I. 
$$
Then $p_{(v_0,2)}\in I$. Similarly we can show that $p_{(v_1,1)}\in I$ using that $p_{(v_1,0)}\in I$ and so on. We obtain that $p_w\in I$ for all $w\in G^0$, hence $I=C^*(G)$ and $p_{(v_0,0)}+p_{(v_1,0)}$ is a full projection. By \cite[Corollary 2.6]{B77} $(p_{(v_0,0)}+p_{(v_1,0)})C^*(G)(p_{(v_0,0)}+p_{(v_1,0)})$ is stably isomorphic to $C^*(G)$. 

We apply the collapse move defined in \cite{S13} a number of times until we obtain a finite graph with no sinks and sources such that every vertex is the base of at least one loop. The graphs below indicate how to obtain such a graph, the vertex indicated with $\ast$ is the one we collapse in each step. The graph we obtain in the last step is denoted by $E$. 

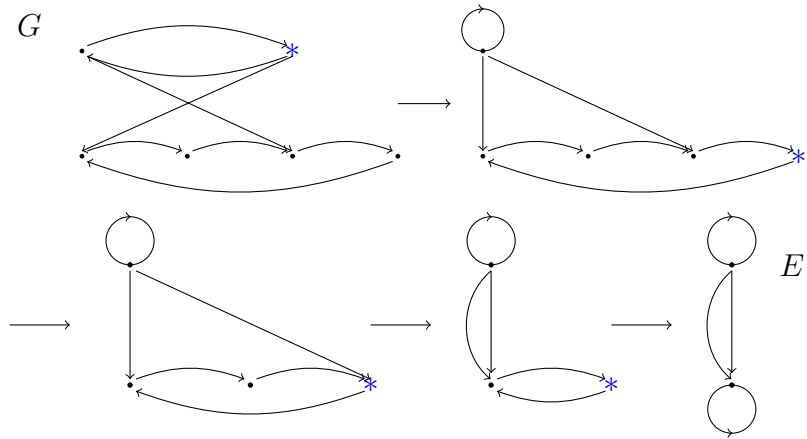
\begin{center}
\begin{tikzpicture}[scale=0.7]
\filldraw [black] (0,0) circle (1pt);
\node at (4,0) {\textcolor{blue}{$\ast$}}; 

\draw[->] (3.9,-0.1) to [out=-160,in=-20] (0.1,-0.1);
\draw[->] (0.1,0.1) to [out=20,in=160] (3.9,0.1);

\node at (-1,0.5)  {$G$};

\filldraw [black] (0,-2) circle (1pt);
\filldraw [black] (2,-2) circle (1pt);
\filldraw [black] (4,-2) circle (1pt);
\filldraw [black] (6,-2) circle (1pt);

\draw[->] (0.1,-1.9) to [out=20,in=160] (1.9,-1.9);
\draw[->] (2.1,-1.9) to [out=20,in=160] (3.9,-1.9);
\draw[->] (4.1,-1.9) to [out=20,in=160] (5.9,-1.9);
\draw[->] (5.9,-2.1) to [out=-160,in=-20] (0.1,-2.1);

\draw[->] (0.1,-0.1) to (3.99,-1.9);
\draw[->] (4,-0.1) to (0,-1.9);

\draw[->] (6,-1) to (7,-1);
\end{tikzpicture}
\begin{tikzpicture}[scale=0.7]
\filldraw [black] (0,0) circle (1pt);
\draw (0,0.4) circle [radius=0.4cm];
\draw [->] (-0.01,0.8) -- (0.01,0.8);

\filldraw [black] (0,-2) circle (1pt);
\filldraw [black] (2,-2) circle (1pt);
\filldraw [black] (4,-2) circle (1pt);
\node at (6,-2) {\textcolor{blue}{$\ast$}};

\draw[->] (0.1,-1.9) to [out=20,in=160] (1.9,-1.9);
\draw[->] (2.1,-1.9) to [out=20,in=160] (3.9,-1.9);
\draw[->] (4.1,-1.9) to [out=20,in=160] (5.9,-1.9);
\draw[->] (5.9,-2.1) to [out=-160,in=-20] (0.1,-2.1);

\draw[->] (0.1,-0.1) to (3.99,-1.9);
\draw[->] (0,-0.1) to (0,-1.9);
\end{tikzpicture}
\end{center}
\begin{center}
\begin{tikzpicture}[scale=0.8]
\draw[->] (4,-1) to (5,-1); 

\filldraw [black] (0,0) circle (1pt);
\draw (0,0.4) circle [radius=0.4cm];
\draw [->] (-0.01,0.8) -- (0.01,0.8);

\filldraw [black] (0,-2) circle (1pt);
\filldraw [black] (2,-2) circle (1pt);
\node at (4,-2) {\textcolor{blue}{$\ast$}};

\draw[->] (0.1,-1.9) to [out=20,in=160] (1.9,-1.9);
\draw[->] (2.1,-1.9) to [out=20,in=160] (3.9,-1.9);
\draw[->] (3.9,-2.1) to [out=-160,in=-20] (0.1,-2.1);

\draw[->] (0.1,-0.1) to (3.99,-1.9);
\draw[->] (0,-0.1) to (0,-1.9);
\draw[->] (-2,-1) to (-1,-1);

\filldraw [black] (6,0) circle (1pt);
\draw (6,0.4) circle [radius=0.4cm];
\draw [->] (5.99,0.8) -- (6.01,0.8);

\node at (8,-2) {\textcolor{blue}{$\ast$}}; 
\filldraw [black] (6,-2) circle (1pt);

\draw[->] (6.1,-1.9) to [out=20,in=160] (7.9,-1.9);
\draw[->] (7.9,-2.1) to [out=-160,in=-20] (6.1,-2.1);

\draw[->] (6,-0.1) to (6,-1.8);

\draw[->] (6,-0.1) to [out=-140,in=140] (5.99,-1.9);

\draw[->] (8,-1) to (9,-1);

\filldraw [black] (10,0) circle (1pt);
\draw (10,0.4) circle [radius=0.4cm];
\draw [->] (9.99,0.8) -- (10.01,0.8);

\draw (10,-2.4) circle [radius=0.4cm];
\draw [->] (9.99,-2.8) -- (10.01,-2.8);
\filldraw [black] (10,-2) circle (1pt);

\draw[->] (10,-0.1) to (10,-1.8);

\draw[->] (10,-0.1) to [out=-140,in=140] (9.99,-1.9);
%
\draw node at (11,0) { $E$};
\end{tikzpicture}
\end{center}
By \cite[Lemma 5.1]{S13} we obtain $C^*(G)\otimes \K \cong C^*(E)\otimes\K$ and hence 
\begin{equation}\label{getcontradict}
C(L_q^3(4;(2,1)))\otimes \K \cong C^*(E)\otimes \K.
\end{equation}
It is a consequence of the claim, \cite[Theorem 2.2]{BS}, made by Brzeziński and Szymański that $C(L_q^3(4;(2,1)))$ is isomorphic to the graph $C^*$-algebra of the following graph: 
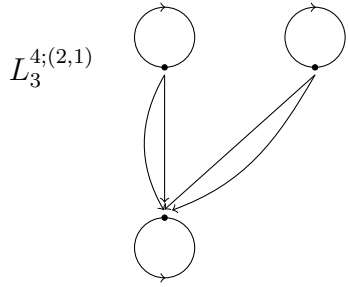
\begin{center}
\begin{tikzpicture}[scale=1]
\filldraw [black] (10,0) circle (1pt);
\draw (10,0.4) circle [radius=0.4cm];
\draw [->] (9.99,0.8) -- (10.01,0.8);

\filldraw [black] (12,0) circle (1pt);
\draw (12,0.4) circle [radius=0.4cm];
\draw [->] (11.99,0.8) -- (12.01,0.8);

\draw (10,-2.4) circle [radius=0.4cm];
\draw [->] (9.99,-2.8) -- (10.01,-2.8);
\filldraw [black] (10,-2) circle (1pt);

\draw[->] (10,-0.1) to (10,-1.8);

\draw[->] (12,-0.1) to (10,-1.89);

\draw[->] (10,-0.1) to [out=-120,in=120] (9.99,-1.9);
\draw[->] (12,-0.1) to [out=-120,in=20] (10.1,-1.9);
%
\draw node at (8.5,0) { $L_{3}^{4;(2,1)}$};
\end{tikzpicture}
\end{center}
By considering the strongly connected components we have $|\text{Prime}_{\gamma}C^*(E)|=2$ and $|\text{Prime}_{\gamma}C^*(L_3^{4;(2,1)})|=3$. Since $\K$ is central and simple it follows that the ideal structure of the tensor product with $\K$ is completely determined by $C^*(E)$ or $C(L_q^3(4;(2,1)))$ (see e.g. \cite[Theorem 4.3.1]{DK94}). Then $C(L_q^3(4;(2,1)))$ cannot be isomorphic to $C^*(L_{3}^{4;(2,1)})$ since it contradicts 
\eqref{getcontradict}. 

We will in Theorem \ref{mainquantumlens} prove that the quantum lens spaces are indeed graph $C^*$-algebras of a modified graph from which it follows that $C(L_q^3(4;(2,1)))\cong C^*(E)$. 

\end{example}

\begin{remark}
The proof of \cite[Theorem 2.2]{BS} follows by constructing an explicit isomorphism. The problem with the isomorphism is that $p_{v_i^b}$ is mapped to $p_{(v_i,b)}$ for $i=0,1,...,n, b=0,1,...,\gcd(m_i,r)-1$. But $p_{(v_i,b)}$ for $b\neq 0$ is not contained in the corner in \eqref{corner} by orthogonality of the projections. 
\end{remark}

\subsection{A modified graph}
We now define a graph for which the main idea behind the construction is similar to the one for ${L}_{2n+1}^{r;\underline{m}}$. The main difference is that we restrict the set of vertices further. 
\begin{definition}
Let $n\geq 1$ be an integer, $r\in \Nb$ and $\underline{m}=(m_0,...,m_n)$ a sequence of positive integers. 
Let $H_{r;\underline{m}}$ be the smallest hereditary subset of $(L_{2n+1}\times_c\Zb_r)^0$ containing $\{(v_i,0)| i=0,...,n\}$. For each $i=0,...,n$ let 
$$
S_i:=\{k\in \{0,...,\gcd(m_i,r)-1\}| \ (v_i,k)\in H_{r;\underline{m}}\}
$$
Note that $S_0=\{0\}$. The graph $\overline{L}_{2n+1}^{r;\underline{m}}$ is defined as follows: 
$$
\begin{aligned}
(\overline{L}_{2n+1}^{r;\underline{m}})^0&:=\{v_i^k| \ i=0,...,n, k\in S_i\}, \\
(\overline{L}_{2n+1}^{r;\underline{m}})^1&:=
\{e_{ij;a}^{st}| 0\leq i\leq j\leq n, s\in S_i, t\in S_j, a=1,...,n_{ij}^{st}\},
\end{aligned}
$$
where $n_{ij}^{st}$ is the number of admissible paths from $(v_i,s)$ to $(v_j,t)$. The range and the source maps are given by: 
$$
s(e_{ij;a}^{st})=v_i^s, \ \ r(e_{ij;a}^{st})=v_j^t.
$$
\end{definition}
The graph $\overline{L}_{2n+1}^{r;\underline{m}}$ then consists of $\sum_{i=0}^n |S_i|$ vertices which we divide into $n+1$ levels. The levels are denoted as level $0$ to level $n$, where level $i$ consists of the vertices $v_i^k, k\in S_i$. There only exist edges from a lower indexed level to a higher one and each vertex is the base of precisely one loop. The graph is illustrated in Figure \ref{illustrationql} without indicating any edges. 

\begin{figure}[H]
\centering
\begin{tikzpicture}[scale=0.7]
\begin{small}
\filldraw [black] (0,-0.5) circle (2pt);
\node at (0.5,-0.5) {$v_0^0$};
\node at (-1.2,-0.5) {Level 0};

\filldraw [black] (0,-1.5) circle (2pt);
\node at (0.5,-1.5) {$v_1^0$};
\node at (-1.2,-1.5) {Level 1};

\filldraw [black] (2,-1.5) circle (2pt);
\node at (2.5,-1.5) {$v_1^1$};

\filldraw [black] (6,-1.5) circle (2pt);
\node at (6.9,-1.5) {$v_1^{|S_1|-1}$};

\draw[dotted] (2,-1.5) to (6,-1.5); 

\draw[dotted] (0,-1.5) to (0,-3); 

\node at (-1.2,-3) {Level n};
\filldraw [black] (0,-3) circle (2pt);
\node at (0.5,-3) {$v_n^0$};
\draw[dotted] (6,-1.5) to (6,-3); 
\draw[dotted] (2,-1.5) to (2,-3); 
\filldraw [black] (2,-3) circle (2pt);
\node at (2.5,-3) {$v_n^1$};

\filldraw [black] (6,-3) circle (2pt);
\node at (6.9,-3) {$v_n^{|S_n|-1}$};

\draw[dotted] (2,-3) to (6,-3); 
\end{small}
\end{tikzpicture}
\label{illustrationql}
\caption{Illustration of $\left(\overline{L}_{2n+1}^{r;\underline{m}}\right)^0$}
\end{figure}
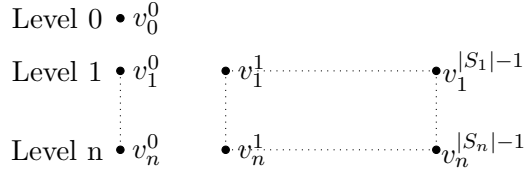

The difference between the definition of $L_{2n+1}^{r;\underline{m}}$ and $\overline{L}_{2n+1}^{r;\underline{m}}$ is that we restrict the vertices to the ones in the smallest hereditary subset of $(L_{2n+1}\times_c\Zb_r)^0$. In this way we avoid the problem in Example \ref{counterexample}, since we remove the vertices which are not in the hereditary subset i.e. the vertices $(v_0,1)$ and $(v_0,3)$. 

The purpose of this section is to prove the following theorem:
\begin{theorem}\label{mainquantumlens}
As $C^*$-algebras we have 
$$
C(L_q^{2n+1}(r;m)) \cong C^*(\overline{L}_{2n+1}^{r;\underline{m}}). 
$$
\end{theorem}
\begin{remark}
\label{Remark:GraphComp}
If $\gcd(m_0,r)=1$, then one can always find a path from $(v_0,0)$ to any given vertex in the skew-product graph, and hence $\overline{L}_{2n+1}^{(r;\underline{m})}$ is the same as $L_{2n+1}^{(r;\underline{m})}$. It follows in this particular case, that our description of quantum lens spaces as graph $C^*$-algebras agrees with the one given in \cite{BS}. Consequently all of the examples given in that paper as well as the work done in \cite{hs} and \cite{errs} is still valid under the description given in Theorem \ref{mainquantumlens}
\end{remark}

To show that $C(L_q^{2n+1}(r;\underline{m}))$ is isomorphic to $C^*(\overline{L}_{2n+1}^{r;\underline{m}})$ we need a couple of lemmas. 
With \cite[Theorem 2.2]{BS} in mind we will show the following.

\begin{lemma}\label{MainLemma}
There exists a $*$-isomorphism 
\begin{equation}\label{cornerisom}
\psi: C^*(\overline{L}_{2n+1}^{r;\underline{m}})\to \left(\sum_{\substack{(v_i,k) \\ i=0,...,n, \ k\in S_i}} p_{(v_i,k)}\right) C^*(L_{2n+1}\times_c\Z_r)\left(\sum_{\substack{(v_i,k) \\ i=0,...,n, \  k\in S_i}} p_{(v_i,k)}\right)
\end{equation}
such that 
$$
\psi(p_{v_i^k})=p_{(v_i,k)}, \ i=0,...,n, k\in S_i. 
$$
For an admissible path $\alpha=(e_{ii_1},k+m_i)(e_{i_1,i_2},k_1)\cdots (e_{i_mi_{m+1}},k_m)(e_{i_{m+1}j},t)$
from $(v_i,k)$ to $(v_j,t)$ with $i,j=0,...,n$, $k\in S_i$ and $t\in S_j$ we let:
$$
\psi(s_{\alpha})=s_{(e_{ii_1},k+m_i)}s_{(e_{i_1,i_2},k_1)}\cdots s_{(e_{i_mi_{m+1}},k_m)}s_{(e_{i_{m+1}j},t)}. 
$$
We denote from now on the corner in \eqref{cornerisom} by $C_{r,\underline{m}}$. 
\end{lemma}
\begin{proof}
Let $\alpha$ and $\beta$ be two admissible paths between vertices in the set $\{(v_i,k)| \ i=0,...,n, k\in S_i\}$. Then $s_{\alpha}^*s_{\beta}=0$ if $\alpha\neq \beta$ since the partial isometries in $C^*(\overline{L}_{2n+1}^{r;\underline{m}})$ have mutually orthogonal range projections. We then have to show that $\psi(s_{\alpha}^*)\psi(s_{\beta})=0$ if $\alpha\neq \beta$. It follows that $\psi(s_{\alpha}^*)\psi(s_{\beta})$ equals $s_{\beta'}$ if $\beta=\alpha\beta'$ and $s_{\alpha'}^*$ if $\alpha=\beta\alpha'$ otherwise it equals zero (see e.g. \cite[Corollary 1.14]{R05}) . Since $\alpha$ and $\beta$ are admissible paths and $\alpha\beta'$ and $\beta\alpha'$ are not the two first cases cannot happen. Hence, $\psi(s_{\alpha}^*)\psi(s_{\beta})$ equals zero if $\alpha\neq \beta$. 

We will now show that the image of $p_{v_i^k}$ and $s_{\alpha}$ satisfies the defining relations in Definition \ref{graphalgebra} for $\overline{L}_{2n+1}^{r;\underline{m}}$. Then by universality $\psi$ is a $*$-homomorphism. Using the defining relations for graph $C^*$-algebras, it follows by an easy calculation that $\psi(s_{\alpha}^*)\psi(s_{\alpha})=p_{r(\alpha)}$ and $\psi(s_{\alpha})\psi(s_{\alpha}^*)\leq p_{s(\alpha)}$, hence condition (CK1) and (CK2) are satisfied. For condition (CK3) we fix a $v_i^k$ in $\left(\overline{L}_{2n+1}^{r;\underline{m}}\right)^0$. Let $A$ be the collection of all admissible paths from $(v_i,k)$ to a $(v_j,t)$ with $j=0,...,n, t\in S_j$. Since each outgoing edge of $v_i^k$ to a $v_j^t$ corresponds to an admissible path from $(v_i,k)$ to $(v_j,t)$, we wish to show that 
\begin{equation}\label{toshow}
p_{(v_i,k)}=\sum_{\alpha\in A} s_{\alpha}s_{\alpha}^*.
\end{equation}
By condition (CK3) on $L_{2n+1}\times_c\Zb_r$ and the identity $s_e=s_ep_{r(e)}=p_{s(e)}s_e$ we have 
\begin{equation}\label{eq1}
\begin{aligned}
p_{(v_i,k)}&=\sum_{i_1=i}^n s_{(e_{ii_1},m_i+k)}s_{(e_{ii_1},m_i+k)}^* \\
&= \sum_{i_1=i}^n s_{(e_{ii_1},m_i+k)}p_{r((e_{ii_1}m_i+k))}s_{(e_{ii_1},m_i+k)}^* \\
&= \sum_{i_1=i}^n s_{(e_{ii_1},m_i+k)}p_{(v_{i_1},m_i+k)}s_{(e_{ii_1},m_i+k)}^*. 
\end{aligned}
\end{equation}
For each $i_1\in \{i,...,n\}$, if $e_{({ii_1},m_i+k)}$ is not in $A$ then we substitute
$$
p_{(v_{i_1},m_i+k)}=\sum_{i_2=i_1}^n s_{(e_{i_1i_2},m_i+m_{i_1}+k)}s_{(e_{i_1i_2},m_i+m_{i_1}+k)}^*
$$
in \eqref{eq1}. Let $A_{1}$ be the set of all $(e_{{ii_1}},m_i+k), i_1=i,...n$ for which $(e_{{ii_1}},m_i+k)$ is inside $A$. Moreover, let $I_1$ be the set of all $i_1$ for which $(e_{{ii_1}},m_i+k)\in A_1$. Note that $A_1$ is finite since the paths have to be admissible. 
Then 
\vspace{0.2cm}
\\
\resizebox{\columnwidth}{!}{
$
\begin{aligned}
p_{(v_i,k)}&= \sum_{\substack{i_1=i \\  i_1\notin I_1}}^n \left( s_{(e_{ii_1},m_i+k)}
\left(\sum_{i_2=i_1}^n s_{(e_{i_1i_2},m_i+m_{i_1}+k)}p_{(v_{i_2},m_i+m_{i_1}+k)}s_{(e_{i_1i_2},m_i+m_{i_1}+k)}^*
\right)
s_{(e_{ii_1},m_i+k)}^*\right) \\
&+ \ \  \sum_{\alpha\in A_1} s_{\alpha}s_{\alpha}^*.
\end{aligned}
$}
\vspace{0.2cm}
\\
 Similarly, let $A_2$ contain all paths $(e_{ii_1},m_i+k)(e_{i_1i_2},m_i+m_{i_1}+k)$ which are contained in $A$. Let $I_2$ be the set of all $(i_1,i_2)$ for which $(e_{ii_1},m_i+k)(e_{i_1i_2},m_i+m_{i_1}+k)\in A_2$. Then 
$$
\begin{aligned}
p_{(v_i,k)}&= \sum_{\substack{i_1=i \\  i_1\notin I_1}}^n \sum_{\substack{i_2=i_1 \\ \ \ \ (i_1,i_2)\notin I_2}}^n s_{(e_{ii_1},m_i+k)}
s_{(e_{i_1i_2},m_i+m_{i_1}+k)} \\
& \ \ \cdot \left( \sum_{i_3=i_2}^n  s_{(e_{i_2i_3},m_i+m_{i_1}+m_{i_2}+k)}s_{(e_{i_2i_3},m_i+m_{i_1}+m_{i_2}+k)}^*   \right)
s_{(e_{i_1i_2},m_i+m_{i_1}+k)}^*
s_{(e_{ii_1},m_i+k)}^* \\
& \ \ +  \sum_{\alpha\in A_1\cup A_2} s_{\alpha}s_{\alpha}^*.
\end{aligned}
$$
Proceeding inductively, let $A_s$ contain the set of all paths
\begin{equation}\label{eq2}
\begin{aligned}
&(e_{ii_1},m_i+k)(e_{i_1i_2},m_i+m_{i_1}+k)(e_{i_2i_3},m_i+m_{i_1}+m_{i_2}+k) \\ &\cdots\cdots
(e_{i_{s-1}i_s},m_i+m_{i_1}+m_{i_2}+\cdots +m_{i_{s-1}}+k)
\end{aligned}
\end{equation}
which are contained in $A$. Let $I_s$ contain all $(i_1,i_2,...,i_s)$ for which the path in \eqref{eq2} is contained in $A_s$. Note that $A_s$ consists of admissible paths of length $s$. 

Continuing as above we will at some point obtain that all the paths are admissible, hence the procedure terminates. This happens since we do not have any edges from a higher level to a lower one in the finite graph $L_{2n+1}\times_c\Zb_r$. Hence, there exists a $m\in\mathbb{N}$ such that $(i_1,i_2,....,i_m)$ are all contained in $I_m$. Then
$$
p_{(v_i,k)}=  \sum_{\alpha\in A_1\cup A_2\cup \cdots \cup A_m} s_{\alpha}s_{\alpha}^*.
$$
Furthermore, since we in each step consider all the outgoing edges of a vertex we construct indeed all admissible paths by this procedure. Hence, $A=A_1\cup A_2\cup \cdots \cup A_m$ and we obtain \eqref{toshow}. 

For surjectivity we observe that 
$$
C_{r,\underline{m}}=\overline{\span}\{s_{\mu}s_{\nu}^*| \ r(\mu)=r(\nu), s(\mu),s(\nu)\in \{(v_i,k), i=0,...,n, k\in S_i\}\},
$$
which follows by the fact that 
$$
\left(\sum_{\substack{(v_i,k) \\ i=0,...,n, \ k\in S_i}} P_{(v_i,k)}\right)s_{\mu}s_{\nu}^*\left(\sum_{\substack{(v_i,k) \\ i=0,...,n, \ k\in S_i}} P_{(v_i,k)}\right)
$$
is non-zero if and only if $s(\mu),s(\nu)\in \{(v_i,k), i=0,...,n, k\in S_i\}$. 

Let $\mu$ and $\nu$ be paths such that $s_{\mu}s_{\nu}^*\in C_{r,\underline{m}}$ and for which the range of $\mu$ and $\nu$ are in $\{(v_i,k), i=0,...,n, k\in S_i\}$. Then we see immediately
that $\mu=\mu_1\cdots\mu_s$ and $\nu=\nu_1\cdots\nu_t$ for some admissible paths $\mu_j,\nu_j$ i.e. $\mu_i,\nu_i$ corresponds to edges in $C^*(\overline{L}_{2n+1}^{r;\underline{m}})$. Then 
$$
\psi(s_{\mu_1\cdots\mu_s}s_{\nu_1\cdots\nu_t}^*)=s_{\mu}s_{\nu}^*.
$$
Hence, $s_{\mu}s_{\nu}^*$ is in the image of $\psi$. 

Let now $\mu$ and $\nu$ be paths with range not in $\{(v_i,k), i=0,...,n, k\in S_i\}$ such that $s_{\mu}s_{\nu}^*\in C_{r,\underline{m}}$.
By using that
$$
s_{\mu}s_{\nu}^*=s_{\mu}\left(\sum_{\substack{e\in (L_{2n+1}\times_c\Zb_r)^1 \\ s(e)=r(\mu)}} s_es_e^* \right)s_{\nu}^*. 
$$
a number of times, similar as in the proof of condition (CK3), we obtain that 
$$
s_{\mu}s_{\nu}^*=\sum_{i=1}^m s_{\mu\alpha_i}s_{\nu\alpha_i}^*
$$
where $\alpha_i, i=1,...,m$ are paths from $r(\mu)$ to a vertex in $\{(v_i,k), i=0,...,n, k\in S_i\}$. Then for $i=1,...,m$, $\mu\alpha_i$ and $\nu\alpha_i$ represents a path in $C^*(\overline{L}_{2n+1}^{r;\underline{m}})$ and $s_{\mu}s_{\nu}^*$ is then in the image of $\psi$. 

Finally, to prove that $\psi$ is injective we apply the generalised Cuntz-Krieger uniqueness theorem presented in \cite[Theorem 1.2]{S02b}. We have that $\psi(p_{v_i^k})$ is non-zero for all $i=0,...,n,k\in S_i$. By \cite[Theorem 1.2]{S02b}, we obtain that $\psi$ is injective if the spectrum of 
$$\psi(s_{e_{nn;1}^{kk}})=s_{(e_{nn},k+m_n)}s_{(e_{nn},k+2m_n)}\cdots s_{\left(e_{nn},k+\left(\frac{r}{\gcd(m_n,r)}-1\right)m_n\right)}s_{(e_{nn},k)}
$$
for each $k\in S_n$ contains the entire unit circle. By the last part of the proof of Theorem 2.4 in \cite{KPR98} it follows that this is indeed the case. Hence, $\psi$ is injective. 
\end{proof}

\begin{lemma}\label{existpath} For each $v_i^k, i=0,...,n, k\in S_i$ there is a path from $v_0^0$ to $v_i^k$ in $\overline{L}_{2n+1}^{r;\underline{m}}$.
\end{lemma}
\begin{proof}
Let $v_i^k, i=0,...,n, k\in S_i$ then $(v_i,k)\in H_{r;\underline{m}}$ by definition. Hence, there exists a path $\alpha$ from $(v_j,0)$ for at least one $j=0,...,n$ to $(v_i,k)$. If the path is not admissible we divide it into admissible subpaths. Furthermore, there is always an admissible path from $(v_{\ell-1},0)$ to $(v_\ell,0)$ for any $\ell=0,...,n$ as follows:
\vspace{0.2cm}
\\
\resizebox{\columnwidth}{!}{
$
(e_{(\ell-1)(\ell-1)},m_{l-1})(e_{(\ell-1)(\ell-1)},2m_{\ell-1})\cdots \left(e_{(\ell-1)(\ell-1)},\left(\frac{r}{\gcd(m_{\ell-1},r)}-1\right)m_{\ell-1}\right)(e_{(\ell-1)l},0).
$}
\vspace{0.2cm}
\\
Combining these paths we obtain a path from $(v_0,0)$ to a $(v_j,0)$ with $j=0,...,n$, call the path $\beta$. Then we obtain a path $\beta\alpha$ from $(v_0,0)$ to $(v_i,k)$  which consists of admissible subpaths. Hence, there is indeed a path from $v_0^0$ to $v_i^k$. 
\end{proof}

\begin{lemma}\label{fullprojection}
The projection $\sum_{i=0}^{n}p_{v_i^0}$ is full in $C^*(\overline{L}_{2n+1}^{r;\underline{m}})$. 
\end{lemma}
\begin{proof}
Let $I$ be the ideal generated by $\sum_{i=0}^{n}p_{v_i^0}$. Note that we clearly have $p_{v_i^0}\in I$ for $i=0,...,n$. We wish to show that for any $v_i^k, i=0,...,n, k\in S_i$ we have $p_{v_i^k}\in I$, since then $I=C^*(\overline{L}_{2n+1}^{r;\underline{m}})$ and $\sum_{i=0}^{n}p_{v_i^0}$ is full. 

Let $\alpha=f_1f_2\cdots f_m$ with $f_j\in (\overline{L}_{2n+1}^{r;\underline{m}})^1$ for $j=1,...,m$ be a path from $v_0^0$ to a $v_i^k$ with $i\in \{0,...,n\},  k\in S_i$ which we know exists by Lemma \ref{existpath}. By the Cuntz-Krieger relations we have 
$$
s_{f_1}s_{f_1}^*\leq p_{v_0^0}, \ \ \ s_{f_1}^*s_{f_1}=p_{r(f_1)}. 
$$
Then $s_{f_1}s_{f_1}^*=s_{f_1}s_{f_1}^*p_{v_0^0}\in I$ and we obtain $s_{f_1}^*\in I$ which implies $p_{r(f_1)}\in I$. We can apply the same argument as above to show $p_{r(f_2)}\in I$ if we replace $p_{v_0^0}$ with $p_{r(f_1)}$ and $p_{r(f_1)}$ with $p_{r(f_2)}$. By continuing this argument we obtain that $p_{v_i^k}=p_{r(f_m)}\in I$. 
\end{proof}

\begin{lemma}\label{graphsisom}
Let $E=(E^0,E^1,r,s)$ be a directed graph. For each $v\in E^0$, choose an $a_v\geq 1$. Define $F$ to be the graph such that for each $v\in E^0$ we add a head of $a_v-1$ vertices to $v$.
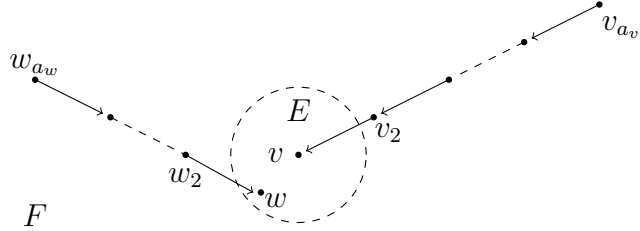
\begin{figure}[H]
\centering
\begin{tikzpicture}[scale=1]
\draw [dashed] (0,0) circle [radius=0.9 cm];
\filldraw [black] (0,0) circle (1pt);
\filldraw [black] (1,0.5) circle (1pt);
\filldraw [black] (2,1) circle (1pt);
\filldraw [black] (3,1.5) circle (1pt);
\filldraw [black] (4,2) circle (1pt);
\filldraw [black] (-0.5,-0.5) circle (1pt);
\filldraw [black] (-1.5,0) circle (1pt);
\filldraw [black] (-2.5,0.5) circle (1pt);
\filldraw [black] (-3.5,1) circle (1pt);
\node at (-0.3,0) {$v$}; 
\node at (-0.3,-0.6) {$w$}; 
\node at (4.3,1.7) {$v_{a_v}$}; 
\node at (-3.5,1.2) {$w_{a_w}$}; 
\node at (1.2,0.3) {$v_{2}$}; 
\node at (-1.5,-0.3) {$w_{2}$}; 
\node at (0,0.6) {$E$}; 
\node at (-3.5,-0.8) {$F$}; 
\draw[->] (1,0.5) -- (0.1,0.05);
\draw[->] (2,1) -- (1.1,0.55);
\draw[dashed ](3,1.5) -- (2.1,1.05);
\draw[->] (4,2) -- (3.1,1.55);

\draw[->] (-1.5,0) -- (-0.6,-0.5);
\draw[dashed] (-2.5,0.5) -- (-1.5,0);
\draw[->] (-3.5,1) -- (-2.6,0.55);
\end{tikzpicture}
\caption{Illustration of the graph $F$}
\label{graphF}
\end{figure}
Define another graph $G$ as follows: 
$$
G^0=E^0\sqcup \{\overline{w}\}, \ \ G^1=E^1\sqcup \{e_{v,i}| \ i=1,...,a_v-1, v\in E^0, a_v\geq 2\}.
$$
The range and the source maps extend from $E$ to $G$ for the edges in $E^1$ and 
$$
s(e_{v,i})=\overline{w}, \ \ r(e_{v,i})=v. 
$$
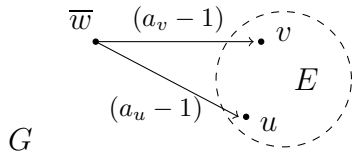
\begin{figure}[H]
\centering
\begin{tikzpicture}[scale=1]
\draw [dashed] (0,0) circle [radius=0.9 cm];
\filldraw [black] (-0.3,0.5) circle (1pt);
\filldraw [black] (-0.5,-0.5) circle (1pt);
\filldraw [black] (-2.5,0.5) circle (1pt);
\node at (0.3,0) {$E$}; 
\node at (-0.2,-0.6) {$u$}; 
\node at (0,0.6) {$v$}; 
\node at (-2.7,0.8) {$\overline{w}$}; 
\node at (-3.5,-0.8) {$G$}; 
\node at (-1.7,-0.4) {\footnotesize $(a_u-1)$}; 
\node at (-1.4,0.8) {\footnotesize $(a_v-1)$};

\draw[->] (-2.5,0.5) to (-0.6,-0.5);
\draw[->] (-2.5,0.5) to (-0.4,0.5);
\end{tikzpicture}
\caption{Illustration of the graph $G$}
\end{figure}
Then as $C^*$-algebras $C^*(F)\cong C^*(G)$. 
\end{lemma}
\begin{proof}
Consider first the graph $F$, we apply the $R^+$-move, presented in \cite[Definition 3.9]{ER19}, on the regular vertices $v_i$, $i=2,...,a_v$ one by one for each $v\in F$. We then obtain a graph $\overline{F}$ defined as follows:
$$
\overline{F}^0=E^0\sqcup \{\tilde{v}_i| \ v\in E^0, i=2,...,a_v\}, \ \ 
\overline{F}^1=E^1\sqcup \{\tilde{e}^v_i| \ v\in E^0, i=2,...,a_v \}
$$
where $s(\tilde{e}^v_i)=\tilde{v}_i$ and $r(\tilde{e}^v_i)=v$. 
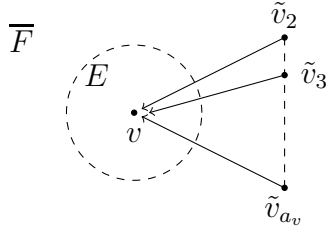
\begin{figure}[H]
\centering \begin{tikzpicture}[scale=1]
\draw [dashed] (0,0) circle [radius=0.9 cm];
\filldraw [black] (0,0) circle (1pt);
\filldraw [black] (2,1) circle (1pt);
\filldraw [black] (2,-1) circle (1pt);
\filldraw [black] (2,0.5) circle (1pt);
\node at (0,-0.3) {$v$}; 
\node at (2,1.3) {$\tilde{v}_2$}; 
\node at (2.4,0.5) {$\tilde{v}_{3}$}; 
\node at (2,-1.3) {$\tilde{v}_{a_v}$}; 
\node at (-0.5,0.5) {$E$}; 
\node at (-1.5,1) {$\overline{F}$}; 
\draw[->] (2,1) -- (0.1,0.05);
\draw[->] (2,-1) -- (0.1,-0.05);
\draw[dashed] (2,1) -- (2,-1);
\draw[->] (2,0.5) -- (0.2,0);
\end{tikzpicture}
\caption{Illustration of the graph $\overline{F}$}
\end{figure}
By \cite[Theorem 3.10]{ER19} $C^*(F)\cong C^*(\overline{F})$. 

For the graph $G$ we perform an out-split (\cite{BP04}, see also \cite[Definition 3.1]{ER19}) of the vertex $\overline{w}$ by partition $s^{-1}(\overline{w})$ into singleton sets. Then the out-split graph $G_O$ is precisely the graph $\overline{F}$. By \cite[Theorem 3.2]{BP04} $C^*(G)\cong C^*(G_O)$, hence $C^*(G)\cong C^*(F)$. 
\end{proof}

We are now ready to prove Theorem \ref{mainquantumlens}.

\begin{proof}[{Proof of Theorem \ref{mainquantumlens}}]
By Lemma \ref{MainLemma} and \eqref{corner} we obtain 
\vspace{0.2cm}
\\
\resizebox{\columnwidth}{!}{
$
\begin{aligned}
&C(L_q^{2n+1}(r;\underline{m})) \cong \left(\sum_{i=0}^n p_{(v_i,0)}\right)C^*(L_{2n+1}\times_{
c}\mathbb{Z}_r)\left(\sum_{i=0}^n p_{(v_i,0)}\right) \\
&= \left(\sum_{i=0}^n p_{(v_i,0)}\right)\left(\sum_{\substack{(v_i,k) \\ i=0,...,n, \  k\in S_i}} p_{(v_i,k)}\right)C^*(L_{2n+1}\times_{
c}\mathbb{Z}_r)\left(\sum_{\substack{(v_i,k) \\ i=0,...,n, \  k\in S_i}} P_{(v_i,k)}\right)\left(\sum_{i=0}^n p_{(v_i,0)}\right)
\\
&\cong \left(\sum_{i=0}^{n}p_{v_i^0}\right)C^*(\overline{L}_{2n+1}^{r;\underline{m}})\left(\sum_{i=0}^{n}p_{v_i^0}\right).
\end{aligned}
$}
\vspace{0.2cm}
\\
Hence, it suffices to prove that 
\begin{equation}\label{corner2}
\left(\sum_{i=0}^{n}p_{v_i^0}\right)C^*(\overline{L}_{2n+1}^{r;\underline{m}})\left(\sum_{i=0}^{n}p_{v_i^0}\right) \cong C^*(\overline{L}_{2n+1}^{r;\underline{m}}). 
\end{equation}
We will follow the procedure in the proof of \cite[Theorem 4.8(1)]{AE15} to construct a graph for which the corner in \eqref{corner2} is isomorphic to its graph $C^*$-algebra. For proofs of the statement used in the construction we refer to \cite[Theorem 4.8(1)]{AE15}. 

For simplicity we let $E:=\overline{L}_{2n+1}^{r;\underline{m}}$ and $P:=\sum_{i=0}^{n}p_{v_i^0}$. Let $SE$ be the graph for which an infinite head has been added to each vertex. $SE$ is called the stabilisation of $E$. 
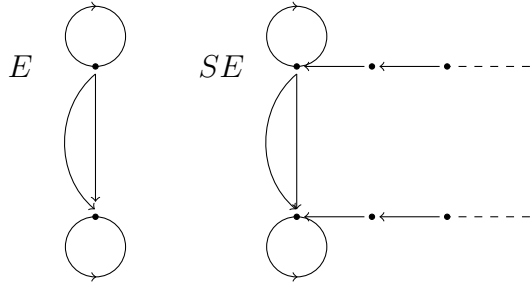
\begin{figure}[H]
\centering
\begin{tikzpicture}[scale=1]

\filldraw [black] (10,0) circle (1pt);
\draw (10,0.4) circle [radius=0.4cm];
\draw [->] (9.99,0.8) -- (10.01,0.8);

\draw (10,-2.4) circle [radius=0.4cm];
\draw [->] (9.99,-2.8) -- (10.01,-2.8);
\filldraw [black] (10,-2) circle (1pt);

\draw[->] (10,-0.1) to (10,-1.8);

\draw[->] (10,-0.1) to [out=-140,in=140] (9.99,-1.9);
%
\draw node at (9,0) { $E$};
\end{tikzpicture}
\ \ \ \ \ 
\begin{tikzpicture}[scale=1]

\filldraw [black] (10,0) circle (1pt);
\draw (10,0.4) circle [radius=0.4cm];
\draw [->] (9.99,0.8) -- (10.01,0.8);

\draw (10,-2.4) circle [radius=0.4cm];
\draw [->] (9.99,-2.8) -- (10.01,-2.8);
\filldraw [black] (10,-2) circle (1pt);

\draw[->] (10,-0.1) to (10,-1.9);

\draw[->] (10,-0.1) to [out=-140,in=140] (9.99,-1.9);

\draw node at (9,0) { $SE$};

\filldraw [black] (11,0) circle (1pt);
\draw[->] (10.9,0) -- (10.1,0);
\filldraw [black] (12,0) circle (1pt);
\draw[->] (11.9,0) -- (11.1,0);
\draw[dashed] (13.1,0) to (12,0); 

\filldraw [black] (11,-2) circle (1pt);
\draw[->] (10.9,-2) -- (10.1,-2);
\filldraw [black] (12,-2) circle (1pt);
\draw[->] (11.9,-2) -- (11.1,-2);
\draw[dashed] (13.1,-2) to (12,-2	); 

\end{tikzpicture}
\caption{The stabilisation when $E=\overline{L}_{3}^{4;(2,1)}$}
\end{figure}
There exists an isomorphism $\phi: C^*\left(E\right)\otimes\mathcal{K}\to C^*(SE)$ such that 
$$K_0(\phi)([p_w\otimes e_{11}])=[p_w]$$
for all $w\in E^0$, where $\{e_{ij}\}$ is a set of matrix units in $\mathcal{K}$, \cite[Proposition 9.8]{AT11}. Then
$$
\begin{aligned}
PC^*(E)P\cong 
(P\otimes e_{11})(C^*(E)\otimes \mathcal{K})(P\otimes e_{11})\cong \phi(P\otimes e_{11})C^*(SE)\phi(P\otimes e_{11}). 
\end{aligned}
$$
Since $P$ is full by Lemma \ref{fullprojection}, it follows by the proof of Proposition 4.7 in \cite{AE15} that there exists a finite hereditary subset $T$ of $(SE)^0$, which contains $E^0$, such that $\phi(P\otimes e_{11})$ is Murray-von Neumann equivalent to 
$$
P_T:=\sum_{v\in T} p_v. 
$$
We obtain $T$ as follows: By an application of \cite[Lemma 4.3]{AE15} there exist integers $a_v\geq 1$ such that 
$$
[P]=\sum_{v\in E^0} a_v[p_v],
$$
which follows since there is a path from $v_0^0$ to any $v_i^k$ by Lemma \ref{existpath}. 
For a vertex $v\in E^0$, we denote the first $a_{v}-1$ vertices in the infinite head added to $v$ as follows:
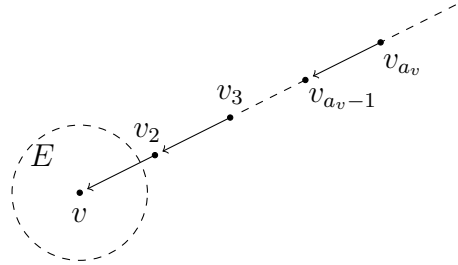
\begin{center}
\begin{tikzpicture}[scale=1]
\draw [dashed] (0,0) circle [radius=0.9 cm];
\filldraw [black] (0,0) circle (1pt);
\filldraw [black] (1,0.5) circle (1pt);
\filldraw [black] (2,1) circle (1pt);
\filldraw [black] (3,1.5) circle (1pt);
\filldraw [black] (4,2) circle (1pt);
\node at (0,-0.3) {$v$}; 
\node at (0.9,0.8) {$v_2$}; 
\node at (2,1.3) {$v_3$}; 
\node at (3.5,1.25) {$v_{a_{v}-1}$}; 
\node at (4.3,1.7) {$v_{a_v}$}; 
\node at (-0.5,0.5) {$E$}; 
\draw[->] (1,0.5) -- (0.1,0.05);
\draw[->] (2,1) -- (1.1,0.55);
\draw[dashed ](3,1.5) -- (2.1,1.05);
\draw[->] (4,2) -- (3.1,1.55);
\draw[dashed] (5,2.5) -- (4,2);
\end{tikzpicture}
\end{center}
Let $v_1=v$ and denote by $e_k$ the edge from $v_k$ to $v_{k-1}$ for $k=2,...,a_v$. 
By the Cuntz-Krieger relations it follows that 
$$
p_{v_{k}}=s_{e_{k}}s_{e_{k}}^*, \ \ p_{v_{k-1}}=s_{e_{k}}^*s_{e_k},
$$
for $k=2,...,a_v$. 
Hence, $[p_{v_k}]=[p_{v_{k-1}}]$ for $k=2,...,a_v$ from which it follows that $[p_{v_{j}}]=[p_v]$ for $j=2,...,a_v$ and 
$$
a_v[p_v]= [p_v]+\sum_{i=2}^{a_v} [p_{v_i}].
$$
Let
$$
T:=E^0 \sqcup \{v_k| \ v\in E^0, a_v\geq 2, k=2,...,a_v\},
$$
then we obtain 
$$
[\phi(P\otimes e_{11})]=[P]=\sum_{v\in E^0} a_v[p_v]= \sum_{v\in E^0}\left([p_v]+\sum_{i=2}^{a_v}[p_{v_i}]\right)=[P_T]. 
$$
Then 
$$
\phi(P\otimes e_{11})C^*(SE)\phi(P\otimes e_{11})\cong P_TC^*(SE)P_T. 
$$ 
By \cite[Theorem 3.15]{AE15} $ P_TC^*(SE)P_T\cong C^*(F)$ where $F=(T,s_{SE}^{-1}(T),r_{SE},s_{SE})$. The graph $F$ consists then of the graph $E$ where for each $v\in E^0$ there is added a head consisting of $a_v-1$ vertices (see Figure \ref{graphF}). We obtain by Lemma \ref{graphsisom} that $C^*(F)$ is isomorphic to $C^*(G)$. Hence, it remains to show that $C^*(G)\cong C^*(E)= C^*(\overline{L}_{2n+1}^{r;\underline{m}})$. 

For this we apply \cite[Theorem 14.6]{errs2} on the graphs $E$ and $G$. First note that $E=\widetilde{E}=\widetilde{G}$, $\underline{x}_E$ is the zero vector of size $|E^0|$ and $\underline{x}_G$ is a vector of size $|E_0|$ which indicates the number of edges from $\overline{w}$ to each vertex in $E^0$. Note that the first entry is $0$, since there are no edges from $\overline{w}$ to $v_0^0$. 

The components of ${E}$ consist of singleton sets. They are all cyclic since every vertex is the base of precisely one loop. 
Recall that the graph $E=\overline{L}_{2n+1}^{r;\underline{m}}$ consists of $n+1$ levels for which we in level $k$ have $|S_k|$ vertices. We denote the vertices in $E^0$ as indicated in Figure \ref{namevertices}. 
\begin{figure}[H]
\centering
\begin{tikzpicture}[scale=0.6]
{\small 
\filldraw [black] (0,-0.5) circle (2pt);
\node at (0,-0.8) {$v_1$};
\node at (-2,-0.5) {Level 0};

\filldraw [black] (0,-2.5) circle (2pt);
\node at (0,-2.8) {$v_2$};
\node at (-2,-2.5) {Level 1};
\draw[dotted] (2,-2.5) to (6,-2.5); 

\filldraw [black] (2,-2.5) circle (2pt);
\node at (2,-2.8) {$v_3$};

\filldraw [black] (6,-2.5) circle (2pt);
\node at (6,-2.8) {$v_{|S_1|+1}$};

\draw[dotted] (2,-4.5) to (6,-4.5); 

\node at (-2,-4.5) {Level 2};
\filldraw [black] (0,-4.5) circle (2pt);
\node at (0,-4.8) {$v_{|S_1|+2}$};

\filldraw [black] (2,-4.5) circle (2pt);
\node at (2,-4.8) {$v_{|S_1|+2}$};

\filldraw [black] (6,-4.5) circle (2pt);
\node at (6,-4.8) {$v_{|S_1|+|S_2|+1}$};

\draw[dotted] (0,-5.3) to (0,-6); 
\draw[dotted] (2,-5.3) to (2,-6);
\draw[dotted] (6,-5.3) to (6,-6); }
\end{tikzpicture}
\caption{Renaming of the vertices in $\overline{L}_{2n+1}^{r;\underline{m}}$}
\label{namevertices}
\end{figure}
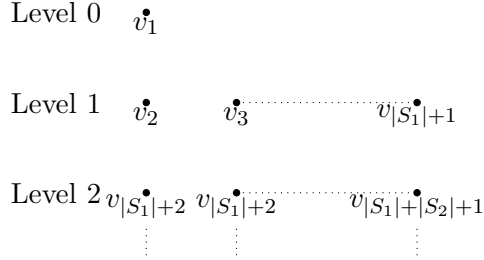

A partial order on the set $\mathcal{P}=\{1,2,...,N\}$ with $N:=\sum_{i=0}^{n} |S_i|$ is defined as follows: For $i,j\in \mathcal{P}$ we let $i\preceq j$ if there is a path from $v_i$ to $v_j$. Let $\gamma_i:=\{v_i\}\in \Gamma_E$ then the map $i\mapsto \gamma_i$ is clearly an order reversing isomorphism from $\mathcal{P}$ to $\Gamma_E$. Furthermore, if $i \preceq j$ then $i\leq j$ which is required. Then $B_{E}=B_{\widetilde{E}}=B_{\widetilde{G}}$ are contained in $\mathfrak{M}^{\circ\circ\circ}_{\mathcal{P}}(\boldsymbol{1},\Zb)$ (see definition in \cite[p.321]{errs}) and 
$$
V_{c_2,...,c_{|E_0|}}=\begin{pmatrix}
1 & c_2 & \cdots & c_{|E^0|} \\
0 & & & \\
\vdots &  & I &  \\
0 &  & & 
\end{pmatrix}
$$
lies inside $\SL_{\mathcal{P}}(\boldsymbol{1},\Zb)$ for any $c_j\in\mathbb{Z}$, since there is a path from $v_0^0$ to every other vertex by Lemma \ref{existpath}. Moreover, we have $B_{\widetilde{E}}V_{c_2,...,c_{|E_0|}}=B_{\widetilde{G}}$ and $\underline{x}_G=\begin{pmatrix}
0 & k_2 & \cdots & k_{|E^0|}
\end{pmatrix}^T
$
with $k_i=a_{v_i}-1$. Then
$$
V_{k_2,...,k_{|E_0|}}^T(\underline{1}+\underline{x}_E)=\begin{pmatrix}
1 \\ k_2+1 \\ \vdots \\ k_{|E_0|}+1 
\end{pmatrix}
=\underline{1}+\underline{x}_G.
$$
Hence, $V_{k_2,...,k_{|E_0|}}^T(\underline{1}+\underline{x}_E)-(\underline{1}+\underline{x}_G)=\underline{0}$ which is clearly in the image of $B_{\widetilde{F}}^T$. Then by letting $U=P=I$ in \cite[Theorem 14.6]{errs2} we obtain $C^*(G)\cong C^*(E)$. 

To summarize we have shown 
$$
\begin{aligned}
C^*(L_{2n+1}^{r;\underline{m}})&\cong C^*(G)\cong C^*(F)\cong P_TC^*(SE)P_T \\
&\cong \phi(P\otimes e_{11})C^*(SE)\phi(P\otimes e_{11}) \\
&\cong PC^*(L_{2n+1}^{r;\underline{m}})P  \cong C(L_q^{2n+1}(r;\underline{m}))
\end{aligned}
$$
which proves the theorem. 
\end{proof}

\section{A classification result of quantum lens spaces}\label{section:classificationQLS}
We will in this section investigate quantum lens spaces $C(L_q^7(r,\underline{m}))$ for which $\gcd(m_\ell,r)=K$ for a single $\ell \in \lbrace 0,1,2,3 \rbrace$ and the remaining weights are coprime to $r$. In the process of finding an invariant for 7-dimensional quantum lens spaces we will also be able to calculate one for quantum lens spaces of dimension 5. We will in the following therefore have our focus on 7-dimensional quantum lens spaces. 

Under the above assumptions on the weights the skew product graph $L_{7}\times_{c}\Z_r$ consists of four levels, labeled by level 0,1,2 and 3, with edges going from level $i$ to $j$ if $i<j$. At the level on which $\gcd(m_\ell,r)=K$ we have $K$ cycles based on each of the vertices $(v_\ell,k), k=0,1,...,K-1$.
The graph $\overline{L}_7^{r;\underline{m}}$ consists of four levels as before with $4$ vertices when $\gcd(m_0,r)=K$ and $K+3$ vertices when $\gcd(m_i,r)=K, i\neq 0$, which are all the base of a loop. There is one vertex in each level except for level $i$ where $\gcd(m_i,r)=K\neq 1, i\neq 0$, here we have $K$ vertices. When $i=0$ we have one vertex in each level. There is at least one edge going from a lower level to a higher one, but there are no edges between vertices at the same level. We will denote the vertices by $v_i, i=1,...,K+3$ as indicated in Figure \ref{namevertices}.  


 The partial order on $\Gamma_{\overline{L}_7^{r;\underline{m}}}$ is given as follows: let $\gamma_i:=\{v_i\}$ then $\gamma_j\leq \gamma_i$ if there is a path from $v_i$ to $v_j$. 
The set $\Gamma_{\overline{L}_7^{r;\underline{m}}}$ can be illustrated by its component graphs, which are depicted in Figure \ref{Fig:Component}. In these graphs, an arrow from $\gamma_i$ to $\gamma_j$ indicates that $\gamma_i\geq \gamma_j$. 
\begin{figure}[H]
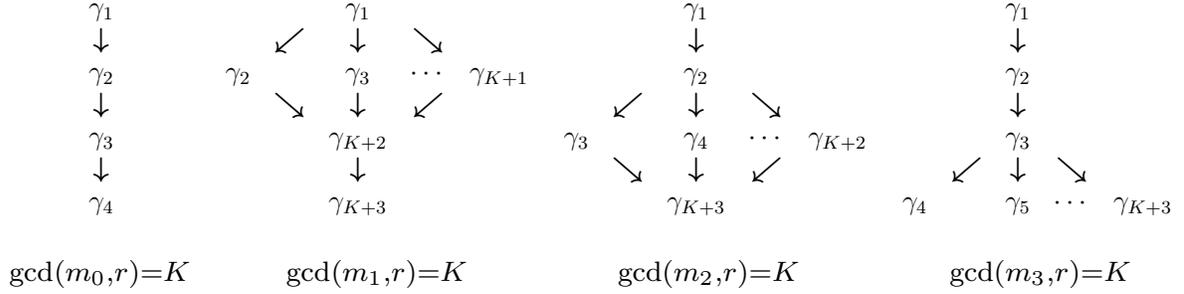

\resizebox{\columnwidth}{!}{
$
\hspace{1cm}
\begin{matrix}
 &  & \gamma_1 &  \\
 & & \boldsymbol{\downarrow}&  &\\
& & \gamma_{2} & & \\
& & \boldsymbol{\downarrow} & & \\
& & \gamma_{3} & &\\
 & & \boldsymbol{\downarrow} & &\\
& & \gamma_{4} & & 
\end{matrix}
\hspace{1cm}
\begin{matrix}
& & \gamma_{1} & & \\
& \boldsymbol{\swarrow}  & \boldsymbol{\downarrow}  & \boldsymbol{\searrow} & \\
\gamma_{2} & & \gamma_3 & \cdots & \gamma_{K+1} \\
& \boldsymbol{\searrow} & \boldsymbol{\downarrow}  & \boldsymbol{\swarrow} & \\
& & \gamma_{K+2} & &  \\
& & \boldsymbol{\downarrow}  & & \\
& & \gamma_{K+3} & & \\
\end{matrix}
\ \ \ \
\begin{matrix}
& & \gamma_{1} & & \\
& & \boldsymbol{\downarrow}  & & \\
& & \gamma_{2} & & \\
& \boldsymbol{\swarrow}  & \boldsymbol{\downarrow}  & \boldsymbol{\searrow} & \\
\gamma_{3} & & \gamma_4 & \cdots & \gamma_{K+2} \\
& \boldsymbol{\searrow} & \boldsymbol{\downarrow}  & \boldsymbol{\swarrow} & \\
& & \gamma_{K+3} & & 
\end{matrix}
\ \ \ \
\begin{matrix}
& & \gamma_1 & & \\
& & \boldsymbol{\downarrow} & & \\
& & \gamma_2 & &\\
 & & \boldsymbol{\downarrow} & &\\
& & \gamma_3 & & \\ 
& \boldsymbol{\swarrow} & \boldsymbol{\downarrow}  & \boldsymbol{\searrow} &\\ 
\gamma_4 & & \gamma_5 & \cdots & \gamma_{K+3}
\end{matrix}
$}
\vspace{0.3cm}
\\
\resizebox{\columnwidth}{!}{$
\ \ \ {\scriptstyle \gcd(m_0,r)=K} \ \ \ \ \ \ \  {\scriptstyle \gcd(m_1,r)=K} \ \ \ \ \ \ \ \ \ \ \ {\scriptstyle \gcd(m_2,r)=K} \ \ \ \ \ \ \ \ \ \ \ {\scriptstyle \gcd(m_3,r)=K } \ \ \
$}
\caption{Component graphs of 7-dimensional quantum lens spaces}
\label{Fig:Component}
\end{figure}

For $\gcd(m_0,r)=K$ we let $\mathcal{P}=\{1,2,3,4\}$ and the partial ordering is linear. 
When $\gcd(m_i,r)=K, i\neq 0$ we let $\mathcal{P}=\{1,2,...,K+3\}$ and $\ell$ be such that $\gcd(m_{\ell-1},r)=K$. We define a partial order, $\preceq$, on $\mathcal{P}$ by:
\begin{itemize}
    \item $\ell-1\succeq...\succeq 1$,
    \item $i\succeq \ell-1$ for $i=\ell,..,K+\ell-1$,
    \item $K+\ell\succeq i$ for $i=\ell,...,K+\ell-1$,
    \item $K+3 \succeq ... \succeq K+\ell$. 
\end{itemize}
The partial order satisfies that if $i\preceq j$ then $i\leq j$ which is the required assumption \cite[Assumption 4.3]{errs}.
It can easily be seen that there exists an order reversing isomorphism $\gamma_{B_{\overline{L}_7^{r;\underline{m}}}}:\mathcal{P}\to \Gamma_{\overline{L}_7^{r;\underline{m}}}$ mapping $i$ to $\gamma_i$. 

Let $\gcd(m_i,r)=K$ for one $i$. It follows immediately that $|\text{Prime}_\gamma(C^*(\overline{L}_7^{r;\underline{m}}))|=K+3$ when $i\neq 0$ and $|\text{Prime}_\gamma(C^*(\overline{L}_7^{r;\underline{m}}))|=4$ if $i=0$. By this result we obtain non isomorphic quantum lens spaces for different values of $K$ when $i\neq 0$. Moreover, by Remark \ref{remark:rinvariant} it follows that $K$ must also be the same in order to obtain isomorphic quantum lens spaces when $i=0$ even though this is not immediately clear by considering the ideal structure.

Fix a $K>1$, by considering the ideal structure we obtain four different isomorphism classes of quantum lens spaces, one for each $i=0,1,2,3$ for which $\gcd(m_i,r)=K$. We will in this section determine when two quantum lens spaces inside each of these four classes are isomorphic. 
Similarly we have three different classes of quantum lens spaces of dimension $5$ to investigate. 

To determine whether two quantum lens spaces in the same class are isomorphic we will make use of \cite[Theorem 7.1]{errs}. For the quantum lens spaces we are investigating, the block structure consists of $1\times 1$-matrices and $(B_{\overline{L}_7^{r;\underline{m}}},B_{\overline{L}_7^{r;\underline{n}}})$ is in standard form for two quantum lens spaces in the same class when we order the vertices in the adjacency matrix as described above. Since $\overline{L}_{2n+1}^{r;\underline{m}}$ contains no sinks Theorem \ref{Theorem:Classification} boils down to: 
\begin{corollary}\label{lensclassification}
Let $\underline{m}=(m_0,m_1,m_2,m_3)$ and $\underline{n}=(n_0,n_1,n_2,n_3)$ be in $\mathbb{N}^4$. If $\gcd(m_i,r)=\gcd(n_i,r)$ for each $i=0,1,2,3$ then $C^*(\overline{L}_7^{r;\underline{m}})$ and $C^*(\overline{L}_7^{r;\underline{n}})$ are isomorphic if and only if there exists matrices $U,V\in SL_{\mathcal{P}}(\mathbf{1},\Zb)$ such that $UB_{\overline{L}_7^{r;\underline{m}}}V=B_{\overline{L}_7^{r;\underline{n}}}$. 
\end{corollary}
We remark that the natural generalisation of Corollary \ref{lensclassification} to quantum lens spaces of other dimensions is true, if one defines the partial order in the obvious way. In particular, for dimension $5$ we have a similar result by letting $\mathcal{P}=\{1,2,...,K+2\}$ and defining the order in a similar way as the one for dimension $7$.
\\

Eilers, Restorff, Ruiz and Sørensen used Corollary \ref{lensclassification}, with $\mathcal{P}=\{1,2,3,4\}$ ordered linearly to completely classify the simplest case: 
\begin{theorem}\label{errsLens}\cite[Theorem 7.8]{errs}
Let $r\in \Nb$, $r\geq 2$ and let $\underline{m}=(m_0,m_1,m_2,m_3)$ and $\underline{n}=(n_0,n_1,n_2,n_3)$ be in $\Nb^4$ such that $\gcd(m_i,r)=\gcd(n_i,r)=1$ for all i. Then $C^*\left(\overline{L}_7^{(r,\underline{m})}\right)\cong C^*\left(\overline{L}_7^{(r,\underline{n})}\right)$ if and only if 
\[(n_2^{-1}n_1-m_2^{-1}m_1)\frac{r(r-1)(r-2)}{3}\equiv 0 \Mod{r}.\]
\end{theorem}
From the above they concluded: 
\begin{corollary}\cite[Corollary 7.9]{errs}
If $3$ does not divide $r$ then 
\[
C^*\left(\overline{L}_7^{(r,\underline{m})}\right)\cong C^*\left(\overline{L}_7^{(r,(1,1,1,1))}\right)
\]
for all $\underline{m}\in \Nb^4$ with $\gcd(m_i,r)=1$. 

If $3$ divides $r$ and $\underline{m}=(m_0,m_1,m_2,m_3)\in\Nb^4$ with $\gcd(m_i,r)=1$ then 
\begin{itemize}
    \item[(i)] $C^*\left(\overline{L}_7^{(r,\underline{m})}\right)\cong C^*\left((\overline{L}_7^{(r,(1,1,1,1))}\right)$ if and only if $m_1\equiv m_2 \Mod{3}$, 
    \item[(ii)] $C^*\left(\overline{L}_7^{(r,\underline{m})}\right)\cong C^*\left(\overline{L}_7^{(r,(1,1,r-1,1))}\right)$ if and only if $m_1\not\equiv m_2 \Mod{3}$. 
\end{itemize}
\end{corollary}
For dimension less than $7$, Eilers, Restorff, Ruiz and Sørensen observed that the adjacency matrices are independent of the weights, hence all quantum lens spaces are isomorphic. We will see that this is not always the case when one of the weights is not coprime with $r$. 

\section{Classification of $C\left(\overline{L}_{2n+1}(r;\underline{m})\right), n\leq 3$}\label{classification}
In this section we fix the value of the order of the acting group, $r$, even though, as we will show later in Remark \ref{remark:rinvariant}, $r$ is in fact an invariant in most of the cases considered. Also, note that for the remainder of this paper, we will use the notation $\mathbb{Z}^\times_k$ to denote the subgroup of multiplicative units of the ring $\mathbb{Z}_k$, and we will use $\phi$ to denote Euler's totient function, that is $\phi(k)\coloneqq \left\vert\mathbb{Z}_k^\times\right\vert$.

Quantum lens spaces of dimension $3$, with one and only one weight coprime with $r$, will be the same for any choice of weights. See remark \ref{dim35}. For dimension $5$ we obtain the following:

\begin{theorem}\label{5dimensional}
Let $r\in \Nb$, $r\geq 2$ and let $\underline{m}=(m_0,m_1,m_2)$ and $\underline{n}=(n_0,n_1,n_2)$ be in $\mathbb{N}^3$ such that $\gcd(m_\ell,r)=\gcd(n_\ell,r)=K$ for one $0\leq \ell \leq 2 $, and $\gcd(m_i,r)=\gcd(n_i,r)=1$ whenever $i\neq \ell$. Then 
\begin{itemize}

    \item[(i)]  $C^*\left(\overline{L}_5^{(r,\underline{m})}\right)\cong C^*\left(\overline{L}_5^{(r,\underline{n})}\right)$ if $\ell\in\lbrace 0,1\rbrace$ and
    
    \item[(ii)] $C^*\left(\overline{L}_5^{(r,\underline{m})}\right)\cong C^*\left(\overline{L}_5^{(r,\underline{n})}\right)$ if and only if $m_1\equiv n_1 \Mod{K}$ if $\ell=2$. 
\end{itemize}
\end{theorem}
The proof of Theorem \ref{5dimensional} follows by a similar, but much easier, approach as the corresponding results in the 7-dimensional case which is presented in Theorem \ref{Thm:Main} below. Therefore we will not present the proof in the present paper.

\begin{corollary}
Let $\gcd(m_2,r)=K$ and $\gcd(m_i,r)=1$ for $i=0,1$, then
\[
C^*\left(L_5^{(r;(m_0,m_1,m_2))}\right)\cong C^*\left(L_5^{(r;(1,k_1,K))}\right),
\]
where $m_1\equiv k_1 \Mod{K}$ with $k_1\in \mathbb{Z}^\times_r$, and there are exactly $\phi(K)$ isomorphism classes of quantum lens spaces. 
\end{corollary}

We will now state our main theorem for quantum lens spaces of dimension $7$, which is an extension of Theorem \ref{errsLens}.

\begin{theorem}
\label{Thm:Main}
Let $r\in \Nb$, $r\geq 2$ and let $\underline{m}=(m_0,m_1,m_2,m_3)$ and $\underline{n}=(n_0,n_1,n_2,n_3)$ be in $\mathbb{N}^4$ such that $\gcd(m_\ell,r)=\gcd(n_\ell,r)=K$ for one $0\leq \ell \leq 3 $, and $\gcd(m_i,r)=\gcd(n_i,r)=1$ whenever $i\neq \ell$. Then $C^*\left(\overline{L}_7^{(r,\underline{m})}\right)$ is isomorphic to $C^*\left(\overline{L}_7^{(r,\underline{n})}\right)$ if and only if
\begin{enumerate}
\item[(0)] $\left(m_2^{-1}m_1-n_2^{-1}n_1 \right) \frac{r(r-1)(r-2)}{3} \equiv 0 \pmod{r}$ if $\ell=0$ and $3\nmid K$. If $3\mid K$ then they are always isomorphic.
    \item[(1)] $\left(n_2^{-1}n_1-m_2^{-1}m_1 \right)\frac{r(r-1)(r-2)}{3}\equiv 0\Mod{r}$ and $m_2\equiv n_2 \pmod{K}$ if $\ell=1$,
    \item[(2)] $\left(n_1^{-1}n_2-m_1^{-1}m_2 \right)\frac{r(r-1)(r-2)}{3}\equiv 0\Mod{r}$ and $m_1\equiv n_1 \pmod{K}$ if $\ell=2$,
    \item[(3)] $\left(n_2^{-1}n_1-m_2^{-1}m_1 \right)\frac{r(r-1)(r-2)}{3}\equiv 0\Mod{r}$ and $m_j\equiv n_j \pmod{K}, j=1,2$ if $\ell=3$.
\end{enumerate}
\end{theorem}

The proof is postponed to section \ref{invariant}. From Theorem \ref{Thm:Main}, we may derive the following results. They are in particular interesting for computational purposes, and gives a precise determination of how many different spaces we obtain of each type. It also shows that the first case is somewhat degenerate.
\begin{corollary}\label{corollary:isomclass}
Let $r\in\Nb$, $\underline{m}=(m_0,m_1,m_2,m_3)\in\Nb^4$ and $\gcd(m_i,r)=K$ for precisely one $i$ and $\gcd(m_j,r)=1$ when $j\neq i$. 
Furthermore, let $k_1,k_2\in \Nb$ be such that $0<k_1,k_2<K$ and $\gcd(k_1,K)=\gcd(k_2,K)=1$. 
\\
If $3\mid K$ or $3\nmid r$ then $C(L_q(r;\underline{m}))$ is isomorphic to a quantum lens space with precisely one of the following sets of weights:
$$
(K,1,1,1), \quad (1,K,k_2,1), \quad (1,k_1,K,1), \quad (1,k_1,k_2,K).
$$
If $3\nmid K$ then $C(L_q(r;\underline{m}))$ is isomorphic to a quantum lens space with one of the following set of weights: 
$$
(K,1,1,1), \quad (1,K,k_2,1), \quad (1,k_1,K,1), \quad (1,k_1,k_2,K) 
$$
if $m_1\equiv m_2 \pmod{3}$ and 
$$
(K,1,r-1,1), \quad (1,K,k_2(r-1),1), \quad (1,k_1(r-1),K,1), \quad (1,k_1,k_2(r-1),K)
$$
if $m_1\not\equiv m_2 \pmod{3}$. 
\end{corollary}
\begin{proof}
First note that in the case were $\gcd(m_0,r)=K$ and $3\nmid r$ the invariant coincides with the analoguous invariant in \cite[Corollary 7.9]{errs}, and we may make the same conclusion. 

We now address the proof in the case where $\gcd(m_3,r)=K$ since the remaining follow by a similar approach.
If $3$ does not divide $r$ we notice that we by Theorem \ref{Thm:Main} only need to consider the condition $m_i\equiv n_i \Mod{K}$. It is clear that if $\gcd(m_i,r)= 1$, then also $\gcd(m_i,K)=1$. Thus, it suffices to show that if $[k]_K\in \mathbb{Z}^\times_K$, then $[k]_K$ contains an element, $k'$, such that $\gcd(k',r)=1$, since then each isomorphism class may be determined by a pair of units in $\mathbb{Z}_K$. Indeed, let $[k]_K\in \mathbb{Z}_K^\times$ be arbitrary. We set $p$ to be the product of $1$ and all prime factors of $r$ which are factors of neither $k$ nor $K$. Now set $k'\coloneqq Kp+k\equiv k \pmod{K}$ and assume that $\gcd(k',r)\neq 1$. Then there exists a common prime factor $q$ of $r$ and $k'$. Since $q$ divides $k'$ but only divides exactly one of $Kp$ and $k$ by construction, we have a contradiction, and $\gcd(k',r)=1$ as desired.

If $3$ divides $r$ then we also need to consider the first part of the invariant. We observe the congruence
\[
\frac{r(r-1)(r-2)}{3}\equiv \frac{2r}{3} \Mod{r}.
\]
If $m_i\equiv \ell_i \pmod{K}$ for $i=1,2$ with $0<\ell_i<K$ and $\gcd(\ell_i,K)=1$ then $3$ must divide $\ell_2^{-1}\ell_1-m_2^{-1}m_1$ to get isomorphic quantum lens spaces. Also notice, that it follows by a computation that $m_1\equiv m_2 \Mod{3}$ if and only if $\ell_1\equiv \ell_2 \Mod{3}$. Hence $C(L_q(r;\underline{m}))$ is isomorphic to $C(L_q(r;(1,\ell_1,\ell_2,K)))$ where $m_1\equiv m_2 \Mod{3}$ if and only if $\ell_1\equiv \ell_2 \Mod{3}$. 
\\ Let $\ell_1\not\equiv \ell_2 \Mod{3}$, we claim that there exists $k_1,k_2$ with $0<k_i<K$ and $\gcd(k_i,K)=1$ for $i=1,2$ such that 
$$
\ell_2\equiv k_2(r-1) \Mod{K}, \hspace{0.2cm} \ell_1\equiv k_1 \Mod{K},  \hspace{0.2cm} k_1\equiv k_2 \Mod{3}.
$$
Let $k_1:=\ell_1$. Assume first that $k_1\in [1]_3$ then $\ell_2\in [2]_3$. Let $k_2=1$ then 
$$
k_2(r-1)\equiv -k_2 \Mod{3} \equiv 2 \Mod{3} \equiv \ell_2 \Mod{3}
$$
and clearly $k_2\equiv k_1 \Mod{3}$. If $k_1\in [2]_3$ then $\ell_2\in [1]_3$ and since $3|r$ we can by the first part of the proof (by letting $K=3$) find $k_2\in [2]_3$ such that $\gcd(k_2,r)=1$ and hence $\gcd(k_2,K)=1$. Then 
$$
k_2(r-1)\equiv -k_2 \Mod{3} \equiv 1 \Mod{3} \equiv k_2 \Mod{3},
$$
and we have proven the claim. Hence, $C(L_q(r;(1,\ell_1,\ell_2,K)))$ is isomorphic to \\ $C(L_q(r;(1,k_1,k_2(r-1),K)))$ for $k_1,k_2$ such that $0<k_i<K$, $\gcd(k_i,K)=1$ and $k_1\equiv k_2 \Mod{3}$.

We will end the proof by showing that there is no overlap between the classes i.e we have to show that 
$$
k_2\equiv k_1 \Mod{3} \Leftrightarrow k_2(r-1)\not\equiv k_1 \Mod{3}
$$
when $0<k_i<K$ and $\gcd(k_i,K)=1$ for i=1,2. Assume that $k_2\equiv k_1 \Mod{3}$ then $k_2(r-1)\equiv -k_1 \Mod{3}$ but $k_1\not\equiv -k_1 \Mod{3}$ since $3\nmid k_1$. The converse follows from a similar line of reasoning. 

\end{proof}
From Corollary \ref{corollary:isomclass}, it follows immediately the we can count the number of isomorphism classes as given in table \ref{Table:Isomorphismclasses}.
\begin{table}[H]
\centering
\begin{tabular}{ c| c c c c c }
 & $3 \not| r$ & & $3|r$ and $3\not| K$ & & $3|r$ and $3|K$\\
\hline
& & & \\
$\gcd(m_0,r)=K$ & 1 & & 2 & & 1 \\
& & & \\
$\gcd(m_1,r)=K$ & $\phi(K)$ & & $2\phi(K)$ & & $\phi(K)$ \\
& & & \\
$\gcd(m_2,r)=K$ & $\phi(K)$ & & $2\phi(K)$ & & $\phi(K)$ \\
& & & \\
$\gcd(m_3,r)=K$ & \phantom{ll}$\phi(K)^2$ & & \phantom{ll}$2\phi(K)^2$ & & \phantom{ll}$\phi(K)^2$  \\
& & & \\

\end{tabular}
\vspace{0.3cm}
\caption{The number of isomorphism classes.}
\label{Table:Isomorphismclasses}
\end{table}

\section{Adjacency matrices}\label{adjacencymatrix}
For each $i=0,...,n$ let $(L_{2n+1}\times_{c}\Zb_r)\left< i\right>$ be the subgraph of $L_{2n+1}\times_{c}\Zb_r$ with vertex set $\{v_i\}\times \Zb_r$ and edges $\{e_{ii}\}\times \Zb_r$. 
\begin{definition}\cite[Definition 7.4]{errs}
 We call an admissible path $(e_{i_1,j_1},h_1)\cdots (e_{i_\ell,j_\ell},h_\ell)$ in $L_{2n+1}\times_{c} \Zb_r$ $k$-step if there exists integers $0<t_1<t_2<\cdots <t_{k+1}$ such that $t_1=i_1$ and $t_{k+1}=j_\ell$ and for each $2\leq \alpha\leq k$ we have 
$$
 \{r((e_{i_s,j_s},h_s))| 1\leq s\leq \ell\}\cap ((L_{2n+1}\times_{c}\Zb_r)\left< t_\alpha \right>)^0\neq \emptyset, 
 $$
 and 
 $$
\{r((e_{i_s,j_s},h_s))| 1\leq s\leq \ell\}\subseteq \bigcup_{i=1}^{k+1} ((L_{2n+1}\times_{c}\Zb_r)\left< t_i \right>)^0.
$$
\end{definition}
Intuitively an admissible path is $k$-step if it touches vertices from precisely $k-1$ different levels not including the level the path starts at and ends in. Below, we give formulae for the number of 1-step, 2-step and 3-step admissible paths in each relevant case. For paths that only touch levels for which the corresponding weights are coprime to the order of the acting group, we refer to \cite[Lemma 7.6]{errs}, which describes this case to completion.

\begin{lemma}{\textbf{1-step admissible paths}}\label{Lemma:1step}
Let $r\in \Nb$, $r\geq 2$ and let $\underline{m}=(m_0,m_1,m_2,m_3)\in \Nb^4$ be such that for a single $\ell\in \lbrace 0,1,2,3\rbrace$, $\gcd(m_\ell,r)=K$ and $\gcd(m_i,r)=1$ for $i\neq \ell$. Let $t\in \lbrace 0,\dots, K-1 \rbrace$ then 
\begin{enumerate}
    \item For $i<\ell$, there are $\frac{r}{K}$ 1-step admissible paths from $(v_i,0)$ to $(v_\ell,t)$.
    \item For $i>\ell$, there are $\frac{r}{K}$ 1-step admissible paths from $(v_\ell,t)$ to $(v_i,0)$.
\end{enumerate}
\begin{proof}
This follows since at the $\ell$'th level we have $K$ loops each going through only one of the $(v_\ell,i)$, $i=0,1,...,K-1$.
\end{proof}
\end{lemma}

\begin{notation}\label{notation:ai}
Consider a quantum lens space with set of weights $\underline{m}=(m_0,m_1,m_2,m_3)$. We will in the following, if $\gcd(m_i,r)= 1$, let $m_i^{-1}$ denote the fixed representative of the equivalence class $[m_i]_r^{-1}$ in $\Zb_r^\times$ for which $0< m_i^{-1} \leq r-1$, and if $K|r$ for some $K\in\Zb$, we denote by $a_i$ a fixed representative of $[m_i]^{-1}_K$ in $\Zb_K^\times$ for which $0< a_i\leq K-1$.
\end{notation}

\begin{lemma}{\textbf{2-step admissible paths}}\label{2step}
Let $r\in \Nb$, $r\geq 2$ and let $\underline{m}=(m_0,m_1,m_2,m_3)\in \Nb^4$ be such that for a single $\ell\in \lbrace 0,1,2,3\rbrace$, $\gcd(m_\ell,r)=K$ and $\gcd(m_i,r)=1$ for $i\neq \ell$. Let $t\in \lbrace 0,\dots, K-1 \rbrace$.
\begin{enumerate}
    \item If $i<k<\ell-1$, there are $\frac{r(r-K)}{2K}+\left(a_k\cdot t-1+q_t K\right)\frac{r}{K}$ 2-step admissible paths from $(v_i,0)$ to $(v_\ell,t)$ passing through the $k$'th level if $t\neq 0$ and $\frac{r(r+K-2)}{2K}$ if $t=0$;
    
    \item If $i>k>\ell-1$, there are $\frac{r(r-K)}{2K}-\left(a_k\cdot t-1+q_tK\right)\frac{r}{K}$ 2-step admissible paths from $(v_\ell,t)$ to $(v_i,0)$ passing through the $k$'th level if $t\neq 0$ and $\frac{r(r-K)}{2K}$ if $t=0$;
    \item If $i<\ell<j$, there are $\frac{r(r-K)}{2K}$ paths from $(v_i,0)$ to $(v_j,0)$ passing through the $\ell$'th level;
\end{enumerate}
where $a_k$ is defined in Notation \ref{notation:ai} and $q_t\in \Z$ is such that $a_kt-1+q_tK$ is an integer between $0$ and $K-1$. 
\begin{proof}
$(1)$ First note that there is only one path from $(v_i,0)$ to $(v_k,sm_k)$ for $s=1,2,...,r-1$ not coming back to $(v_i,0)$ and not going through any vertices at the k-th level. We have an edge from $(v_k,sm_k)$ into the cycle containing $(v_\ell,t)$ if and only if $
m_k(s+1)\equiv t \Mod{K}.$
Equivalently, $s\equiv a_kt-1 \Mod{K}$. Let $q_t\in \Z$ be such that $s_t:=a_kt-1+q_tK$ is an integer between $0$ and $K-1$. Hence $(v_k,s_tm_k)$ is the first vertex in level k which has a path ending in the cycle containing $(v_{\ell},t)$. The number of paths from $(v_k,sm_k)$ for $s=1,2,...,r-1$ to $(v_{\ell},t)$ is then 
\[
\begin{aligned}
&\frac{r-(s-h)}{K}, \ \text{if} \ s\equiv h=0,...,s_t \Mod{K}, \\
& \frac{r-(s-h)}{K}-1, \ \text{if} \ s\equiv h= s_t+1,...,K-1\Mod{K}. 
\end{aligned}
\]
The number of 2-step admissible paths then becomes 
\[
\begin{aligned}
&\sum_{h=0}^{s_t}\sum_{\substack{s=1, \\ s\equiv h \Mod{K}}}^{r-1} \frac{r-(s-h)}{K} + \sum_{h=s_t+1}^{K-1} \sum_{\substack{s=1, \\ s\equiv h \Mod{K} }}^{r-1} \left(\frac{r-(s-h)}{K}-1\right) \\
&
=\sum_{h=0}^{K-1}\sum_{\substack{s=1, \\ s\equiv h \Mod{K}}}^{r-1} \frac{r-(s-h)}{K}- \sum_{h=s_t+1}^{K-1} \sum_{\substack{s=1, \\ s\equiv h \Mod{K} }}^{r-1}1 \\
&= \sum_{j=1}^{\frac{r-K}{K}} n\frac{r-jK}{K} + (K-1)\frac{r}{K} - (K-1-s_t)\frac{r}{K} \\
&=  \frac{r(r-K)}{2K} + (a_kt-1+q_tK)\frac{r}{K}.
\end{aligned}
\]
If $t=0$ then $q_0=1$ hence the number of 2-step admissible paths becomes
\[
\frac{r(r-K)}{2K}+(K-1)\frac{r}{K}=\frac{r^2-rK+2Kr-2r}{2K}=\frac{r(r+K-2)}{2K}.
\]
(2) Let $t>0$. First, we have precisely one path from $(v_{\ell},t)$ to $(v_k,m_kh)$ whose first vertex in the $k$'th level is $(v_k,m_kh)$ if and only if $m_kh\equiv t \Mod{K}$ for $h=1,2,\dots,r-1$. The number of paths from $(v_k,m_kh)$ to $(v_i,0)$ is $r-h$. Hence the total number of admissible 2-step paths is 
\vspace{0.2cm}
\\
\resizebox{\columnwidth}{!}{$
\begin{aligned}
\sum_{\substack{h=1, \\ h\equiv a_kt \Mod{K}}}^{r-1} (r-h)
&= \sum_{j=0}^{\frac{r-K}{K}} (r-(a_kt+q_tK+jK)) = \frac{r(r+K)}{2K}-(a_kt+q_tK)\frac{r}{K} \\
&= \frac{r(r-K)}{2K}+\frac{r}{K}-(a_kt+q_tK)\frac{r}{K}
=\frac{r(r-K)}{2K}-(a_kt-1+q_tK)\frac{r}{K},
\end{aligned}$}
\vspace{0.2cm}
\\
where $q_t\in\mathbb{Z}$ is such that $0<a_kt+q_tK<K$. If $t=0$ then the number of 2-step admissible paths becomes 
\[
\sum_{j=1}^{\frac{r-K}{K}}(r-jK)=\frac{r(r-K)}{2K}. 
\]
(3) Note that for each $t\in \lbrace 1,\dots,K-1\rbrace$ and for each $h\in\lbrace  1, \dots, \frac{r}{K}-1 \rbrace$, there is precisely one edge from $(v_i,0)$ to $(v_\ell,t+m_\ell h)$, and the number of admissible paths from $(v_\ell,t+m_\ell h)$ to $(v_j,0)$ is $\frac{r}{K}-h$. Thus the number of admissible 2-step paths is

\[\sum_{t=0}^{K-1}\sum_{h=1}^{\frac{r}{K}-1} \frac{r}{K}-h = \frac{r(r-K)}{2K}. \]
\end{proof}
\end{lemma}
\begin{remark}\label{dim35}
For quantum lens spaces of dimension 3 (that is $n=1$), we see immediately by Lemma \ref{2step} that the adjacency matrices for a fixed $r$ and $K$, will all be the same. For quantum lens spaces of dimension $5$ the adjacency matrices are given by the following which are obtained by adding the number of 1-step and 2-step admissible paths as found in Lemma \ref{Lemma:1step}, \ref{2step} and \cite[Lemma 7.6]{errs}: 
\vspace{0.2cm}
\begin{center}
\begin{tabular}{c c c c}
$\gcd(m_0,r)=K$ & $\gcd(m_1,r)=K$ & &$\gcd(m_2,r)=K$\\ & & & \\
$\begin{pmatrix}
1 & \frac{r}{K} & \frac{r}{K}+\frac{r(r-K)}{2K} \\
0 & 1 & r \\
0 & 0 & 1
\end{pmatrix}$  &
$\begin{psmallmatrix}
1 & \frac{r}{K} & \frac{r}{K} & \dots & \frac{r}{K} & \frac{r(r+K)}{2K} \\
0 &  &  &  &  & \frac{r}{K}  \\
0 &  &  &  &  & \frac{r}{K} 
\\
0 & & & I_n  & & \vdots 
\\
\vdots & & & & & \frac{r}{K} 
\\
0 & 0 & 0  & \dots &  0 &  1 
\end{psmallmatrix}$ & &
$\begin{psmallmatrix}
1 & r & z_0 & \cdots & \cdots & z_{K-1}\\
0 & 1 & r/K & \cdots & \cdots & r/K \\
0 & 0 &  &  &  & \\
\vdots & \vdots &  & I_{n} &  &  \\
0 & 0 &  &  &  &  \\
 0 & 0 &  &  &  & 
\end{psmallmatrix}$
\end{tabular}
\end{center}
\vspace{0.2cm}
where $z_0=\frac{r(r+K)}{2K}$, $z_t=\frac{r(r-K)}{2K}+\frac{r}{K}a_1t+q_tr$ for $t=1,\dots,K-1$ and $a_1$ is defined in Notation \ref{notation:ai} .
\end{remark}

We will now calculate the adjacency matrices for $7$-dimensional quantum lens spaces. 
Since the following pages are rather heavy on similarly looking matrices and formulae, we provide in Table \ref{Table:MatrixOverview} an overview of where the adjacency matrix of each of our four cases may be found in the following pages, both in terms of page numbering and Lemma numbering.

\begin{table}[H]
    \centering
    \begin{tabular}{|c|c|c|}
    Case & Lemma & Page \\ \hline
    $\gcd(m_0,r)= K$   &\ref{Lemma:m0}  &\pageref{Lemma:m0} \\
      $\gcd(m_1,r)= K$    &\ref{lemmam1}  &\pageref{lemmam1} \\
        $\gcd(m_2,r)=K$     &\ref{lemmam2}  &\pageref{lemmam2} \\
            $\gcd(m_3,r)=K$     &\ref{Lemma:m3}  &\pageref{Lemma:m3}
    \end{tabular}
    \caption{Reference table for adjacency matrices of 7-dimensional spaces}
    \label{Table:MatrixOverview}
\end{table}
\begin{lemma}\label{Lemma:m3} Let $r\in \Nb$, $r\geq 2$ and let $\underline{m}=(m_0,m_1,m_2,m_3)$ be such that $\gcd(m_3,r)=K$ and $\gcd(m_i,r)=1,i\neq 3$. Then we may for each $0\leq \ell < r-1$ and each $0\leq t \leq K-1$ find $k_\ell,s_t\in \Zb$ such that
\[
A_{\overline{L}_q^{7}(r;\underline{m})}= 
\begin{psmallmatrix}
1 & r & \frac{r(r+1)}{2} & x_0 & \cdots & \cdots & x_{K-1} \\
0 & 1 & r & y_0 & \cdots & \cdots & y_{K-1}\\
0 & 0 & 1 & r/K & \cdots & \cdots & r/K  \\
0 & 0 & 0 &  &  &  & \\
\vdots & \vdots & \vdots &  & I_{K} &  &  \\
0 & 0 & 0 &  &  &  &  \\
0 & 0 & 0 &  &  &  & 
\end{psmallmatrix}
\]
where
\[y_0=\frac{r(r+K)}{2K}\]
and
\[
x_0\equiv -m_2^{-1}m_1\frac{r(r-2)(r-1)}{3K}+\sum_{\ell=1}^{r-2} \ell\frac{r(1-k_\ell)}{K} +\sum_{h=0}^{K-1}\sum_{\mathclap{{\substack{\ell=1 \\\ell\equiv m_2a_1h-1\Mod{K}}}}}^{r-2}\frac{\ell h}{K} - \frac{r}{K}\Mod{r}. 
\]
For $t\geq 1$ we have 
\[y_t=\frac{r(r-K)}{2K}+a_2t\frac{r}{K} + rq_t\]
and
\[\begin{aligned}x_t&\equiv-m_2^{-1}m_1\frac{r(r-2)(r-1)}{3K} + \sum_{\ell=1}^{r-2} \ell\frac{r(1-k_\ell)}{K} \\ &+\mathop{\sum_{h=0}^{K-1}\sum_{\ell=1}^{r-2}}_{\ell\equiv m_2a_1h-1\Mod{K}}\frac{\ell h}{K}  - \mathop{\sum_{h=s_t+1}^{K-1}\sum_{\ell=1}^{r-2}}_{\ell\equiv m_2a_1h-1\Mod{K}}\ell +\frac{r}{K}\left(a_2t+a_1t-1 \right) \Mod{r}
\end{aligned}
\]
where $a_i$ is defined in Notation \ref{notation:ai}.
\begin{proof}
We will now calculate the number of 3-step admissible paths from $(v_0,0)$ to $(v_3,t)$ for $t=0,1,2,3$. First notice that there are exactly $\ell$ paths from $(v_0,0)$ to $(v_1,\ell m_1)$ not coming back to $(v_0,0)$, and exactly one edge from $(v_1,\ell m_1)$ to $(v_2,(\ell+1)m_1)$, hence we wish to find the number of paths from $(v_2,(\ell+1)m_1)$ to $(v_3,t)$, denoted $P_\ell$. Then the total number of 3-step admissible paths will be $\sum_{\ell=1}^{r-2}\ell P_\ell$.  

We can express $(v_2,(\ell+1)m_1)$ as $(v_2,sm_2)$, where $0<s\leq r$ satisfies $m_2s\equiv (\ell+1)m_1 \Mod{r}$. As in the proof of Lemma \ref{2step}(1), we let $s_t$ denote a representative of the class $[a_2t-1]_K$ which lies between $0$ and $K-1$, and let $k_\ell$ be an integer such that $0 <m_2^{-1}m_1(\ell+1)+rk_\ell<r $. By reasoning as in Lemma \ref{2step} (1) we have 
\[
P_\ell=\frac{r-(m_2^{-1}m_1(\ell+1)+rk_\ell-h)}{K}
\]
if $s\equiv h=0,...,s_t \Mod{K}$ i.e. $\ell\equiv m_2a_1h-1 \Mod{K}$ and 
\[
P_\ell=\frac{r-(m_2^{-1}m_1(\ell+1)+rk_\ell-h)}{K}-1
\]
if $s\equiv h=s_t+1,...,K-1 \Mod{K}$. The number of 3-step admissible paths becomes 
\vspace{0.2cm} 
\\
\resizebox{\columnwidth}{!}{$
\begin{aligned}
\sum_{\ell=1}^{r-2}\ell P_\ell&=
\sum_{h=0}^{s_t}\sum_{\mathclap{{\substack{\ell=1 \\\ell\equiv m_2a_1h-1\Mod{K}}}}}^{r-2}\ell\frac{r-(m_2^{-1}m_1(\ell+1)+k_\ell r-h)}{K} +
\sum_{h=s_t+1}^{K-1}\sum_{\mathclap{{\substack{\ell=1 \\\ell\equiv m_2a_1h-1\Mod{K}}}}}^{r-2}\ell\left(\frac{r-(m_2^{-1}m_1(\ell+1)+k_\ell r-h)}{K}-1\right) \\
&=\sum_{h=0}^{K-1}\sum_{\mathclap{{\substack{\ell=1 \\\ell\equiv m_2a_1h-1\Mod{K}}}}}^{r-2}\ell\frac{r-(m_2^{-1}m_1(\ell+1)+k_\ell r-h)}{K} - \sum_{h=s_t+1}^{K-1}\sum_{\mathclap{{\substack{\ell=1 \\\ell\equiv m_2a_1h-1\Mod{K}}}}}^{r-2}\ell
\\
&=-\sum_{h=0}^{K-1}\sum_{\mathclap{{\substack{\ell=1 \\ \ell\equiv m_2a_1h-1\Mod{K}}}}}^{r-2}\frac{m_2^{-1}m_1(\ell+1)l}{K} + \sum_{h=0}^{K-1}\sum_{\mathclap{{\substack{\ell=1 \\\ell\equiv m_2a_1h-1\Mod{K}}}}}^{r-2}\ell\frac{r(1-k_\ell)}{K}  +\mathop{\sum_{h=0}^{n-1}\sum_{\ell=1}^{r-2}}_{\ell\equiv m_2a_1h-1\Mod{K}}\frac{\ell h}{K}  - \mathop{\sum_{h=s_t+1}^{K-1}\sum_{\ell=1}^{r-2}}_{\ell\equiv m_2a_1h-1\Mod{K}}\ell
\\
&=-\sum_{\ell=1}^{r-2}\frac{m_2^{-1}m_1(\ell+1)l}{K} + \sum_{\ell=1}^{r-2}\ell\frac{r(1-k_\ell)}{K}  +\mathop{\sum_{h=0}^{K-1}\sum_{\ell=1}^{r-2}}_{\ell\equiv m_2a_1h-1\Mod{K}}\frac{\ell h}{K}  - \mathop{\sum_{h=s_t+1}^{K-1}\sum_{\ell=1}^{r-2}}_{\ell\equiv m_2a_1h-1\Mod{K}}\ell
\\
&= -m_2^{-1}m_1\frac{r(r-2)(r-1)}{3K}+\sum_{\ell=1}^{r-2} \ell\frac{r(1-k_\ell)}{K} +\mathop{\sum_{h=0}^{K-1}\sum_{\ell=1}^{r-2}}_{\ell\equiv m_2a_1h-1\Mod{K}}\frac{\ell h}{K}  - \mathop{\sum_{h=s_t+1}^{K-1}\sum_{\ell=1}^{r-2}}_{\ell\equiv m_2a_1h-1\Mod{K}}\ell. 
\end{aligned}
$}
\vspace{0.2cm}
\\
Adding up the 1-step, 2-step and 3-step admissible paths we arrive at the above adjacency matrix, here we also make use of \cite[Lemma 7.6 (i), (ii)]{errs}. 
\end{proof}
\end{lemma}

\begin{lemma}
\label{Lemma:m0}
Let $r\in \Nb$, $r\geq 2$ and let $\underline{m}=(m_0,m_1,m_2,m_3)$ be such that $\gcd(m_0,r)=K$ and $\gcd(m_i,r)=1, i\neq 0$. Then we may for each $0\leq l < r-1$ find $k_\ell,\in \Zb$ such that
\[
A_{\overline{L}_q^{7}(r;\underline{m})}=
\begin{psmallmatrix}
 1& \frac{r}{K} & y_0 & x_0
\\   0 &   1 & r & \frac{r(r+1)}{2}
\\   0 &  0 & 1 & r
\\    0 & 0 & 0 & 1
\end{psmallmatrix}
\]
where 
\[y_0=\frac{r(r-K)}{2K}+\frac{r}{K}\]
and 
\[
\begin{aligned}
x_0&\equiv -m_2^{-1}m_1\frac{r(r-2)(r-1)}{3K}+ \sum_{\ell=0}^{r-2} \ell\frac{r(1-k_\ell)}{K} \\ &-\sum_{h=0}^{K-1} \sum_{\mathclap{{\substack{\ell=1 \\\ell\equiv a_1h\pmod{K}}}}}^{r-2} \frac{a_1h+Kq_h}{K}\left(r-m_2^{-1}m_1(\ell+1)-rk_\ell \right) \Mod{r}.
\end{aligned}
\]
where $a_1$ is defined in Notation \ref{notation:ai}.
\end{lemma}
\begin{proof} 
Notice that we only need to calculate the number of 3-step admissible paths from $(v_0,t)$ to $(v_3,0)$, where $t=0$, but for the sake of completeness, we will count the remaining paths as well.

For each $0\leq t \leq K-1$ and $1 \leq \ell \leq r-1$, there is a path from $(v_0,t)$ to $(v_1,m_1\ell)$ for which the latter is the first vertex in the first level which is reached by the path, if and only if $m_1\ell \equiv t \pmod{K}$. Consequently, the total number of paths from $(v_0,t)$ to a general $(v_1,m_1\ell)$ not passing through $(v_1,0)$ is the same as the number of $1\leq \tilde{l}\leq \ell$ for which $m_1\tilde{\ell} \equiv t \pmod{K}$. Let $q_\ell$ be the integer such that $0 < a_1h + q_\ell K \leq K$, then $c_t  \coloneqq a_1t+q_tK$ is the least value, $\ell$, for which there is an edge from the cycle starting at $(v_0,t)$ to $(v_1,m_1\ell)$. The number of paths from $(v_0,t)$ to $(v_1,m_1\ell)$ not passing through $(v_1,0)$ is then given as follows:
\[
\frac{l-(a_1h+Kq_h)}{K}
\]
if $\ell m_1\equiv h \Mod{K}$ for $0\leq h<c_t$ and 
\[
\frac{l-(a_1h+Kq_h)}{K}+1
\]
if $\ell m_1\equiv h \Mod{K}$ for $c_t\leq h<K$. 
There is precisely one edge from $(v_1,\ell m_1)$ to $(v_2,m_1(\ell+1))$. We can express $(v_2,l(m_1+1))$ as $(v_2,sm_2)$ i.e. $m_1(\ell+1)\equiv sm_2 \Mod{r}$. Let $k_\ell$ be such that $0<m_2^{-1}m_1(\ell+1)+rk_\ell<K$. Then the number of paths from $(v_1,\ell m_1)$ to $(v_3,0)$ is $r-(m_2^{-1}m_1(\ell+1)+rk_\ell).$ The total number of 3-step admissible paths becomes: 
\[
\begin{aligned}
&\sum_{h=0}^{c_t-1} \sum_{\mathclap{{\substack{\ell=1 \\\ell\equiv a_1h\pmod{K}}}}}^{r-2} \frac{l-(a_1h+Kq_h)}{K}\left(r-m_2^{-1}m_1(\ell+1)-rk_\ell \right)\\ 
&+\sum_{h=c_t}^{K-1} \sum_{\mathclap{{\substack{\ell=1 \\\ell\equiv a_1h\pmod{K}}}}}^{r-2} \left(\frac{\ell-(a_1h+Kq_h)}{K}+1\right)\left(r-m_2^{-1}m_1(\ell+1)-rk_\ell \right) \\
&= -m_2^{-1}m_1\frac{r(r-2)(r-1)}{3K}+ \sum_{\ell=0}^{r-2} \ell\frac{r(1-k_\ell)}{K} \\ &-\sum_{h=0}^{K-1} \sum_{\mathclap{{\substack{\ell=1 \\\ell\equiv a_1h\pmod{K}}}}}^{r-2} \frac{a_1h+Kq_h}{K}\left(r-m_2^{-1}m_1(\ell+1)-rk_\ell \right)+\sum_{h=c_t}^{K-1} \sum_{\mathclap{{\substack{\ell=1 \\\ell\equiv a_1h\pmod{K}}}}}^{r-2} \left(r-m_2^{-1}m_1(\ell+1)-rk_\ell \right) \\
\end{aligned}
\]
For $t=0$ we have $c_t=K$ hence the number of 3-step admissible paths is 
\vspace{0.2cm}
\\
$
-m_2^{-1}m_1\frac{r(r-2)(r-1)}{3K}+ \sum_{\ell=0}^{r-2} \ell\frac{r(1-k_\ell)}{K} -\sum_{h=0}^{K-1} \sum_{\mathclap{{\substack{\ell=1 \\\ell\equiv a_1h\pmod{K}}}}}^{r-2} \frac{a_1h+Kq_h}{K}\left(r-m_2^{-1}m_1(\ell+1)-rk_\ell \right).
$
\end{proof}

\begin{lemma}\label{lemmam2}
Let $r\in \Nb$, $r\geq 2$ and let $\underline{m}=(m_0,m_1,m_2,m_3)$ be such that $\gcd(m_2,r)=K$ and $\gcd(m_i,r)=1, i\neq 2$. Then 
\[
A_{\overline{L}_q^{7}(r;\underline{m})}=
\begin{psmallmatrix}
1 & r & \frac{r(r+n)}{2K} & \frac{r(r-n)}{2K}+ a_1\frac{r}{K} & \dots & \frac{r(r-K)}{2K}+ a_1(K-1)\frac{r}{K}  & x \\
0 & 1 & \frac{r}{K} & \frac{r}{K} & \dots & \frac{r}{K} & \frac{r(r+K)}{2K}\\
0 & 0 & & & & & \frac{r}{K}
\\
\vdots & \vdots & & & I_{K} & & \vdots
\\
0 & 0 &  & & & &\frac{r}{K}
\\
0 & 0 & 0& 0& \cdots & 0 & 1
\end{psmallmatrix}
\]
where
\[
x\equiv -m_1^{-1}m_2\frac{r(2r-K)(r-K)}{6K^2}+m_1^{-1}\frac{r(r-K)(K-1)}{4K}+\frac{r(r-1)}{2} \Mod{r}.
\]
and $a_1$ is defined in Notation \ref{notation:ai}.
\begin{proof}
We will now calculate the number of 3-step admissible paths from $(v_0,0)$ to $(v_0,3)$. First we have $m_1^{-1}m_2\ell-1+tm_1^{-1}+rs_t$ paths from $(v_0,0)$ to each $(v_1,\ell m_2-m_1+t)$ for $t=0,..,K-1$, where $s_\ell$ is such that $0<m_1^{-1}m_2\ell-1+tm_1^{-1}+rs_\ell<r$. The vertex $(v_1,\ell m_2-m_1+t)$ is connected to $(v_2,\ell m_2+t)$ by a single edge and there are $\frac{r}{K}-\ell$ paths from $(v_2,\ell m_2+t)$ to $(v_3,0)$. The total number of 3-step admissible paths becomes 
\[
\begin{aligned}
&\sum_{t=0}^{K-1}\sum_{\ell=1}^{\frac{r-K}{K}} \left(m_1^{-1}m_2\ell-1+tm_1^{-1}+rs_\ell \right) \left( \frac{r}{K}-\ell \right)
\\
&\equiv K \sum_{\ell=1}^{\frac{r-K}{K}} -\ell\left(m_1^{-1}m_2\ell-1 \right)+\sum_{t=0}^{K-1}\sum_{\ell=1}^{\frac{r-K}{K}} tm_1^{-1}\left( \frac{r}{K}-\ell \right) \Mod{r}
\\
&\equiv -m_1^{-1}m_2 \frac{r(2r-K)(r-K)}{6K^2} + \frac{r(r-K)}{2K}+\sum_{t=0}^{K-1} tm_1^{-1}\frac{r(r-K)}{2K^2} \Mod{r}
\\
&\equiv -m_1^{-1}m_2 \frac{r(2r-K)(r-K)}{6K^2} + \frac{r(r-K)}{2K}+ m_1^{-1}\frac{r(r-K)(K-1)}{4K} \Mod{r}.
\end{aligned} 
\]
\end{proof}

\end{lemma}

We state the final case without proof, as the proof is similar to that of Lemma \ref{lemmam2}. 
\begin{lemma}\label{lemmam1}
Let $r\in \Nb$, $r\geq 2$ and let $\underline{m}=(m_0,m_1,m_2,m_3)$ be such that $\gcd(m_1,r)=K$ and $\gcd(m_i,r)=1, i\neq 1$. Then 
\[
A_{\overline{L}_q^{7}(r;\underline{m})}=
\begin{psmallmatrix}
1 & \frac{r}{K} & \frac{r}{K} & \dots & \frac{r}{K} & \frac{r(r+K)}{2K} & x \\
0 &  &  &  &  & \frac{r}{K} & \frac{r(r-n)}{2K}+\frac{r}{K} \\
0 &  &  &  &  & \frac{r}{K} & \frac{r(r-K)}{2K}-\frac{r}{K}a_2+\frac{r}{K}
\\
0 & & & I_K  & & \vdots & \vdots
\\
\vdots & & & & & \frac{r}{K} & \frac{r(r-n)}{2K}-\frac{r}{K}a_2(K-1)+\frac{r}{K}
\\
0 & 0 & 0  & \dots &  0 &  1 & r
\\
0 & 0 & 0& 0& \cdots & 0 & 1
\end{psmallmatrix}
\]
where
\[
x\equiv -m_2^{-1}m_1\frac{r(2r-K)(r-K)}{6K^2}+m_2^{-1}\frac{r(r-K)(K-1)}{4K}+\frac{r(r-1)}{2} \Mod{r}.
\]
and $a_2$ is defined in Notation \ref{notation:ai}.
\end{lemma}

\begin{remark}\label{remark:rinvariant}
It follows directly from the adjacency matrices and Lemma \ref{Lemma:1step}, that $r$ is an invariant for $n>1$ whenever two quantum lens spaces share the same ideal structure and at most one of the weights are not coprime to the order of the acting group. To see this, note that it follows from a simple computation that $SL_\mathcal{P}$-equivalence preserves the elements directly above the main diagonal. Let now $C(\overline{L}_q^{2n+1}(r;\underline{m}))$ and $C(\overline{L}_q^{2n+1}(r';\underline{n}))$ be two isomorphic quantum lens spaces for which only a single weight, say $m_i$ and $n_i$, $i\in\lbrace 0, \dots n \rbrace$ is coprime to the order of the acting group. If $i>0$, the ideal structure guarantees that $\gcd(m_i,r)=\gcd(n_i,r')$ and by the structure of the adjacency matrices we obtain directly that $r=r'$. If $i=0$, then we obtain from the adjacency matrix that  $\frac{r}{\gcd(m_0,r)}=\frac{r'}{\gcd(n_0,r')}$ and $r=r'$. From which it follows that $\gcd(m_0,r)=\gcd(n_0,r')$. 
Hence $r$ and the greatest common divisor are invariant whenever $n>1$.
If $n=1$ and $i=1$ we obtain that $r$ is an invariant by the same argument as for $n>1$ and $i\neq 0$. If $n=1$ and $i=0$ then something else happens. Let $\gcd(m_0,r)=K$ in this case 
$$
A_{\overline{L}_3^{r;(m_0,m_1)}}=\begin{pmatrix}
1 & \frac{r}{K} \\
0 & 1
\end{pmatrix}.
$$
We can then obtain the same $C^*$-algebra for different values of the order of the acting group and the greatest common divisor. For example we have that $A_{\overline{L}_3^{4,(2,1)}}=A_{\overline{L}_3^{2,(1,1)}}$ and hence $C^*(\overline{L}_3^{(2,(1,1))})= C^*(\overline{L}_3^{(4,(2,1))})$.  
\end{remark}

\section{The invariant}\label{invariant}
We are now ready to prove Theorem \ref{Thm:Main}.
\begin{proof}[Proof of Theorem \ref{Thm:Main}]
(0). 
Assume that the quantum lens spaces coming from the weights $\underline{m}=(m_0,m_1,m_2,m_3)$ and $\underline{n}=(n_0,n_1,n_2,n_3)$ are isomorphic. We denote by $x_0,y_0,a_1$ and $x_0',y_0',a_1'$ the elements coming from Lemma \ref{Lemma:m0} corresponding to the system of weights $\underline{m}$ and $\underline{n}$ respectively. See Notation \ref{notation:ai} for the definition of $a_1$.

Consider the following expression
\vspace{0.2cm}\\
\resizebox{\columnwidth}{!}{$
\begin{aligned}
Kx_0 \equiv -m_2^{-1}m_1 \frac{r(r-1)(r-2)}{3} + \sum_{h=0}^{K-1} \sum_{\mathclap{{\substack{\ell=1 \\\ell\equiv a_1h\pmod{K}}}}}^{r-2} {(a_1h+Kq_h)}\left(m_2^{-1}m_1(\ell+1)\right) \Mod{r}. 
\end{aligned}$}
\vspace{0.2cm}\\
Let $b_h$ be such that $1 \leq a_1h+b_hK \leq K$. Then
\vspace{0.2cm}
\\
\resizebox{\columnwidth}{!}{$
\begin{aligned}
\sum_{h=0}^{K-1} \sum_{\mathclap{{\substack{\ell=1 \\\ell\equiv a_1h\pmod{K}}}}}^{r-2} {a_1h+Kq_h}\left(m_2^{-1}m_1(\ell+1)\right) \equiv \sum_{h=0}^{K-1}\left(a_1h+Kq_h\right)m_2^{-1}m_1 \sum_{k=0}^{\frac{r-K}{K}} \left(a_1h+b_hK+kK+1\right) \pmod{r}
\end{aligned}$}
\vspace{0.2cm}\\
remarking that the corresponding terms for \text{$\ell=0$} and \text{$\ell =r-1$} vanish modulo $r$. We have \\
\resizebox{\columnwidth}{!}{$
\begin{aligned}
&\sum_{h=0}^{K-1}\left(a_1h+Kq_h\right)m_2^{-1}m_1 \sum_{k=0}^{\frac{r-K}{K}} \left(a_1h+b_hK+kK+1\right)
\\
= & \sum_{h=0}^{K-1}\left(a_1h+Kq_h\right)m_2^{-1}m_1\left( \frac{r}{K}\left(a_1h+b_hK+1 \right) + \frac{r(r-K)}{2K} \right) 
\\
\equiv & \sum_{h=0}^{K-1}\left(a_1h+Kq_h\right)m_2^{-1}m_1\left( \frac{r}{K}\left(a_1h+1 \right) + \frac{r(r-K)}{2K} \right) \pmod{r}
\\
\equiv &\sum_{h=0}^{K-1}a_1hm_2^{-1}m_1\left( \frac{r}{K}\left(a_1h+1 \right) + \frac{r(r-K)}{2K} \right)+ \sum_{h=0}^{K-1}q_hm_2^{-1}m_1\left( \frac{r(r-K)}{2} \right) \pmod{r}
\\
=&a_1m_2^{-1}m_1 \frac{r(K-1)}{2}\left(1+ \frac{(r-K)}{2}\right)+ a_1^2m_2^{-1}m_1\frac{r(K+1)(2K+1)}{6}+\sum_{h=0}^{K-1}q_hm_2^{-1}m_1\left( \frac{r(r-K)}{2} \right) \pmod{r}
\end{aligned}
$}
Notice, that if the parity of $K$ and $r$ is the same, $r-K$ is even and hence the leftmost and rightmost terms will vanish modulo $r$ when we subtract the corresponding terms of $Kx_0'$. (For the first term, note that then either $K-1$ or $\left(a_1m_2^{-1}m_1 - a_1'n_2^{-1}n_1 \right)$ is even).
Additionally, it follows from a computation that 
$$a_1^2m_2^{-1}m_1\frac{r(K+1)(2K+1)}{6} \equiv a_1m_2^{-1}\frac{r(K+1)(2K+1)}{6} \pmod{r}$$
So if the parity of $r$ and $K$ is the same, we have
\vspace{0.2cm}
\\
\resizebox{\columnwidth}{!}{$
\begin{aligned}
K(x_0'-x_0) \equiv  \left(m_2^{-1}m_1-n_2^{-1}n_1 \right) \frac{r(r-1)(r-2)}{3}+\left(a_1'n_2^{-1}-a_1m_2^{-1}\right)\frac{r(K+1)(2K+1)}{6} \pmod{r}
\end{aligned}$}
We now consider the case where $r$ is even and $K$ is odd. 
We first let 
$$J_0\coloneqq \sum_{h=0}^{K-1}q_hm_2^{-1}m_1\left( \frac{r-K}{K} \right)$$ so that we may rewrite the last term as $J_0K \frac{r}{2}$ and note that $J_0$ is an integer since $K$ divides $r-K$. Consider next the first term. Since $K-1$ is even, we have
\vspace{0.2cm}
\\
\resizebox{\columnwidth}{!}{$
\begin{aligned}
\left(a_1m_2^{-1}m_1-a_1'n_2^{-1}n_1\right) \frac{r(K-1)}{2}\left(1+ \frac{(r-K)}{2}\right) \equiv \left(a_1m_2^{-1}m_1-a_1'n_2^{-1}n_1\right) \frac{r(K-1)}{2}\frac{(r-K)}{2} \pmod{r}\end{aligned}$}
\vspace{0.2cm}
\\
which we may by a similar argument conclude has the form $KJ_1 \frac{r}{2}$ for an integer $J_1$.
It follows that $K(x_0'-x_0)$ is equivalent to
\begin{equation}
\left(m_2^{-1}m_1-n_2^{-1}n_1 \right) \frac{r(r-1)(r-2)}{3}+\left(a_1'n_2^{-1}-a_1m_2^{-1}\right)\frac{r(K+1)(2K+1)}{6} + KJ \frac{r}{2}    \label{EQ:KJ}
\end{equation}
modulo $r$ where $J=J_0-J_0'+J_1$. Note that from the analysis above we can conclude that $KJ\frac{r}{2}\equiv 0\pmod{r}$ if the parity of $r$ and $K$ is the same.
\\
By an application of $SL_\mathcal{P}$-equivalence (Corollary \ref{lensclassification}) we have that there exist integers $v_1,v_2,u$ such that
$$ x_0' \equiv \frac{r}{K}v_1 +v_2y_0+x_0+u\frac{r(r+1)}2 \pmod{r} $$
And hence
$$
\begin{aligned}
K(x_0'-x_0) \equiv v_2 \frac{r(r-K)}{2} +u \frac{r(r+1)K}{2} \pmod{r}
\end{aligned}
$$
If $r$ and $K$ have the same parity then the right-hand side is congruent to zero modulo $r$ and $KJ\frac{r}{2}$ is congruent to zero modulo $r$.

Assume now that $r$ and $K$ have opposite parity (in particular, $K$ is odd and $r$ is even). We then obtain 
$$
\begin{aligned}
\left(m_2^{-1}m_1-n_2^{-1}n_1 \right) \frac{r(r-1)(r-2)}{3}+\left(a_1'n_2^{-1}-a_1m_2^{-1}\right)\frac{r(K+1)(2K+1)}{6} = M \frac{r}{2}
\end{aligned}
$$
where $M=v_2(r-K)+u(r+1)K-KJ$. If $3\nmid r$ then it is easy to see (remarking that the left-hand-side is then an integer multiple of $r$) that $M$ is even, recalling that in this case, either $3\mid (K+1)$ or $3\mid (2K+1)$.
If $3\mid r$, then the expression may be rewritten as
$$N\frac{r}{3}=M\frac{r}{2}$$
for some integer $N$, and consequently
$$M= 2 \frac{N}{3}$$
from which it follows that $M$ is even, and hence
\begin{equation}\resizebox{0.9\columnwidth}{!}{
\label{Eq:M0Inv}
   $ \left(m_2^{-1}m_1-n_2^{-1}n_1 \right) \frac{r(r-1)(r-2)}{3}+\left(a_1'n_2^{-1}-a_1m_2^{-1}\right)\frac{r(K+1)(2K+1)}{6} \equiv 0 \pmod{r}.$}
\end{equation}

Consider the term $$\left(a_1'n_2^{-1}-a_1m_2^{-1}\right)\frac{r(K+1)(2K+1)}{6}.$$
It follows from an easy computation, that if $3\nmid K$ then this is congruent to zero modulo $r$, and consequently the invariant coincides with the analoguous invariant for the other three cases, and we may make the same conclusion as in \cite[Corollary 7.9]{errs}. If $3$ divides both $r$ and $K$, then $a_1 \equiv m_1 \pmod{3}$, $m_2^{-1}\equiv m_2 \pmod{3}$ etc, while neither of $(K+1), (2K+1), (r-1)$ or $(r-2)$ is divisible by 3. So we in this case have
\vspace{0.2cm}
\\
\resizebox{\columnwidth}{!}{
 $\left(m_2^{-1}m_1-n_2^{-1}n_1 \right) \frac{r(r-1)(r-2)}{3}+\left(a_1'n_2^{-1}-a_1m_2^{-1}\right)\frac{r(K+1)(2K+1)}{6} \equiv 0 \pmod{r}$}
 \vspace{0.1cm}
 \\
 if and only if
 $$2\left(m_2^{-1}m_1-n_2^{-1}n_1 \right)(r-1)(r-2) + \left(a_1'n_2^{-1}-a_1m_2^{-1}\right)(K+1)(2K+1)\equiv 0 \pmod{3}$$
 To see the 'if'-part of the statement, one should notice that the left-hand-side of the second equation is always even and a multiple of $3$ under the given assumptions and hence a multiple of 6. 
 Now, since $(r-1)(r-2) \equiv 2 \pmod{3}$ and $(K+1)(2K+1) \equiv 1 \pmod{3}$ and 2 is self-inverse modulo 3, the above is congruent to
  $$\left(m_2^{-1}m_1-n_2^{-1}n_1 \right) + \left(a_1'n_2^{-1}-a_1m_2^{-1}\right)\equiv 0 \pmod{3}$$
which is
  $$m_2^{-1}\left(m_1-a_1\right) + n_2^{-1}\left(a_1'-n_1 \right) \equiv 0 \pmod{3} $$
And since $a_1' \equiv n_1 \pmod{3}$ and $a_1 \equiv m_1 \pmod{3}$ this is true independent of the choice of weights, and hence this case is trivial.

For the reverse implication assume now that \eqref{Eq:M0Inv} holds.
By appealing to \eqref{EQ:KJ}, there is an integer $s$ such that
$$K(x_0'-x_0) = sr + KJ \frac{r}2 \iff x_0'-x_0= s\frac{r}{K}+ J\frac{r}{2}$$
 We will make use of \cite[Proposition 2.14]{errs}. Hence we need to show that we can transform
$$
\begin{psmallmatrix}
 0  &\frac{r}{K} & y_0 & x_0
\\   0  & 0 & r & \frac{r(r+1)}{2}
\\   0 & 0 & 0 & r
\\   0 & 0 & 0 & 0
\end{psmallmatrix}
$$
into
$$
\begin{psmallmatrix}
 0 &\frac{r}{K} & y_0 & x_0'
\\   0  &  0 & r & \frac{r(r+1)}{2}
\\   0  & 0 & 0 & r
\\   0 & 0 & 0 & 0
\end{psmallmatrix}
$$
by adding columns (rows) of either matrix to subsequent (prior) columns (rows) of the same matrix an integral number of times.
This is done by the following operations: First we add the second row to the first row $J$ times to obtain 
$$
x_0+J\frac{r(r+1)}{2}=x_0+J\frac{r}{2}+J\frac{r^2}{2}
$$
We can then transform this into $x_0'$ by adding the second column to the fourth $s-JK\frac{r}{2}$ times (recall that $J$ is even if $r$ and $K$ are both odd by the first part of the proof).
\vspace{0.2cm}
\\
(3).
Assume that the quantum lens spaces coming from the weights $\underline{m}=(m_0,m_1,m_2,m_3)$ and $\underline{n}=(n_0,n_1,n_2,n_3)$ are isomorphic. We denote by $x_0,...,x_{K-1}, y_0,...,y_{K-1},a_1,a_2$ and $x_0',...,x_{K-1}'$, $y_0',...,y_{K-1}',a_1',a_2'$ the elements $x_i$ and $y_i$ coming from Lemma \ref{Lemma:m0} corresponding to $\underline{m}$ and $\underline{n}$ respectively. See Notation \ref{notation:ai} for the definition of $a_i$. Then by Corollary \ref{lensclassification} we obtain the following equations
\begin{equation}\label{matrixequa1}
    y_j'\equiv y_j+u_{23}\frac{r}{K}+rv_{3,j+4}, \ \ \ \
x_j'\equiv x_j+u_{12}y_j+v_{3,j+4}\frac{r(r+1)}{2}+u_{13}\frac{r}{K} \Mod{r},
\end{equation}
for $j=0,1,...,K-1$ where $u_{\ell m}, v_{\ell m}\in\mathbb{Z}$ are the entries of the matrices $U$ and $V$ from the $SL_{\mathcal{P}}$-equivalence.

Note that since $y_0=y_0'$, $K$ divides $u_{23}$. Then by (\ref{matrixequa1})
\[
a_2'\frac{r}{K}\equiv a_2\frac{r}{K}+u_{23}\frac{r}{K} \Mod{r},
\]
hence there exists a $k\in\Z$ such that 
\[
(a_2'-a_2)\frac{r}{K}=u_{23}\frac{r}{K}+kr.
\]
Then $\frac{a_2'-a_2}{K}\in \Z$ and $a_2'=a_2$ which implies that $m_2\equiv n_2 \Mod{K}$. 

We now consider the sum of all the $x_i's$. We have 
\[
\begin{aligned}
x_0+x_1+...+x_{K-1}&\equiv -m_2^{-1}m_1\frac{r(r-2)(r-1)}{3}+ n\sum_{h=0}^{K-1}\sum_{\mathclap{{\substack{\ell=1 \\\ell\equiv m_2a_1h-1\pmod{K}}}}}^{r-2}\frac{\ell h}{K} \\ &- \sum_{t=1}^{K-1} \sum_{h=s_t+1}^{K-1}\sum_{\mathclap{{\substack{\ell=1 \\\ell\equiv m_2a_1t-1\pmod{K}}}}}^{r-2}\ell +\sum_{t=1}^{K-1}\frac{r}{K}\left(a_1t+a_2t) \right)\pmod{r} \\
&\equiv 
-m_2^{-1}m_1\frac{r(r-2)(r-1)}{3}+\sum_{h=0}^{K-1}\sum_{\mathclap{{\substack{\ell=1 \\\ell\equiv m_2a_1h-1\pmod{K}}}}}^{r-2}\ell h \\ &- \sum_{t=1}^{K-1} \sum_{h=s_t+1}^{K-1}\sum_{\mathclap{{\substack{\ell=1 \\\ell\equiv m_2a_1t-1\pmod{K}}}}}^{r-2}\ell 
+ \frac{r(K-1)}{2}(a_1+a_2) \pmod{r}.
\end{aligned}
\]
The last term is always congruent to $0$ modulo $r$, indeed if $K$ is odd we are done, if $K$ is even then $a_1+a_2$ is even. 

Since each $s_t$ corresponds uniquely to a number between $0$ and $K-1$, we may reiterate the penultimate sum accordingly:
\[
\begin{aligned}
\mathop{\sum_{t=1}^{K-1} \sum_{h=s_t+1}^{K-1}\sum_{\ell=1}^{r-2}}_{\ell\equiv m_2a_1h-1\Mod{K}}\ell = \mathop{\sum_{t=1}^{K-1}
\sum_{h=t}^{K-1}\sum_{\ell=1}^{r-2}}_{\ell\equiv m_2a_1h-1\Mod{K}}\ell  
=\sum_{h=1}^{K-1}\sum_{\mathclap{{\substack{\ell=1 \\\ell\equiv m_2a_1h-1\Mod{K}}}}}^{r-2}h\ell. 
\end{aligned}
\]
Hence 
\[
x_0+x_1+...+x_{K-1}\equiv -m_2^{-1}m_1\frac{r(r-2)(r-1)}{3} \Mod{r}.
\]
Using (\ref{matrixequa1}) and the fact that $(a_2-1)\frac{r(K-1)}{2}\equiv 0 \Mod{r}$, we have 
\vspace{0.2cm}
\\
\resizebox{\columnwidth}{!}{$
\begin{aligned}
&(x_0'+x_1'+...+x_{K-1}')-(x_0+x_1+...+x_{K-1})\equiv u_{12}(y_0+y_1+...+y_{K-1})\\ &+Ku_{13}\frac{r}{K} + (v_{34}+v_{35}+...+v_{3,K+3})\frac{r(r+1)}{2} \Mod{r} \\
&\equiv u_{12}\left(\frac{r(r+K)}{2K}+ (K-1)\frac{r(r-K)}{2K}+a_2\frac{r}{K}\sum_{t=1}^{K-1} t\right)+(v_{34}+v_{35}+...+v_{3,K+3})\frac{r(r+1)}{2} \Mod{r}  \\
&\equiv
u_{12}\frac{r(r+1)}{2}+(v_{34}+v_{35}+...+v_{3,K+3})\frac{r(r+1)}{2} \Mod{r} \\
&\equiv
(v_{34}+v_{35}+...+v_{3,K+3})\frac{r(r+1)}{2}\Mod{r} \equiv 0 \Mod{r}.
\end{aligned}
$} 
\vspace{0.2cm}
\\
The last congruence follows by \cite[the proof of Theorem 7.8]{errs}. Hence
\[
\left(m_2^{-1}m_1-n_2^{-1}n_1\right)\frac{r(r-1)(r-2)}{3}\equiv 0 \Mod{r}. \]
It remains to be shown that $m_1\equiv n_1 \Mod{K}$ i.e. $a_1=a_1'$. 
For $h=s_t+1,...,K-1$ let $p_h\in\Z$ be such that $0<m_2a_1h-1+Kp_h<K$. First we need to expand the following sum: 
\vspace{0.2cm}
\\
\resizebox{\columnwidth}{!}{
$
\begin{aligned}
&\mathop{\sum_{h=s_t+1}^{K-1}\sum_{\ell=1}^{r-2}}_{\ell\equiv m_2a_1h-1\Mod{K}}\ell 
= \sum_{h=s_t+1}^{K-1}\sum_{k=0}^{\frac{r-K}{K}}\left( m_2a_1h-1+Kp_h+Kk\right) = \sum_{h=s_t+1}^{K-1}\left(\frac{r}{K}(m_2a_1h-1)+rp_h+\frac{r(r-K)}{2K} \right) \\
&\equiv \frac{r}{K}m_2a_1\left(\frac{K(K-1)}{2}+\frac{a_2t(1-a_2t)}{2}\right)+\frac{r}{K}(a_2t-K)+(K-a_2t)\frac{r(r-K)}{2K} \Mod{r}.
\end{aligned}
$}
\vspace{0.2cm}
\\
We will now find an expression for $(x_0'-x_t')-(x_0-x_t)$ and then consider the expressions for $t=1$ and $t=K-1$. From the above we have
\vspace{0.2cm}
\\
\resizebox{\columnwidth}{!}{
$
\begin{aligned}
(x_0'-x_t')-(x_0-x_t)&\equiv \mathop{\sum_{h=s_t+1}^{K-1}\sum_{\ell=1}^{r-2}}_{\ell\equiv m_2a_1'h-1\Mod{K}}\ell -\mathop{\sum_{h=s_t+1}^{K-1}\sum_{\ell=1}^{r-2}}_{\ell\equiv m_2a_1h-1\Mod{K}}\ell + \frac{r}{K}(a_1-a_1')t \Mod{r} \\
&\equiv \frac{r}{K}\left(\frac{K(K-1)}{2}-\frac{a_2t(a_2t-1)}{2}\right)m_2(a_1'-a_1)+ \frac{r}{K}(a_1-a_1')t \pmod{r}.
\end{aligned}
$}
\vspace{0.2cm}
\\
On the other hand by (\ref{matrixequa1}) we have 
\[
\begin{aligned}
(x_0'-x_t')-(x_0-x_t)&\equiv u_{12}(y_0-y_t)+(v_{34}-v_{3,t+4})\frac{r(r+1)}{2} \pmod{r} \\
&\equiv -u_{12}a_2t\frac{r}{K} \pmod{r},
\end{aligned}
\]
which follows since $v_{34}=-\frac{u_{23}}{K}$. Indeed, we have $y_0'=y_0$ and 
\[
\frac{r(r-K)}{2K}+a_2't\frac{r}{K}=y_t'=rv_{3,t+4}+y_t+u_{23}\frac{r}{K}=rv_{3,t+4}+\frac{r(r-K)}{2K}+a_2t\frac{r}{K}+u_{23}\frac{r}{K}
\]
hence $rv_{3,t+4}+u_{23}\frac{r}{K}=(a_2'-a_2)t\frac{r}{K}=0$ and $v_{3,t+4}=-\frac{u_{23}}{K}$.
\\

Combining the two expressions for $(x_0'-x_t')-(x_0-x_t)$ we get 
\[
\frac{r}{K}\left(\frac{K(K-1)}{2}-\frac{a_2t(a_2t-1)}{2}\right)m_2(a_1'-a_1)+ \frac{r}{K}(a_1-a_1')t+u_{12}a_2t\frac{r}{K}\equiv 0 \Mod{r}.
\]
Note that 
\[
\frac{r}{K}\left(\frac{K(K-1)}{2}\right)m_2(a_1-a_1')=r\left(\frac{K-1}{2}\right)m_2(a_1-a_1')\equiv 0 \mod r,
\]
since if $K$ is even then $2$ divides $a_1-a_1'$ and if $K$ is odd $2$ divides $K-1$. Thus we have
\begin{equation}\label{x1xt}
\left(\frac{r}{K}\frac{a_2t(a_2t-1)}{2}m_2+t\right)(a_1-a_1')+u_{12}a_2t\frac{r}{K}\equiv 0 \pmod{r}.
\end{equation}
By taking the sum of the expressions in (\ref{x1xt}) where we choose $t=1$ and $t=K-1$, we arrive at 
\[
\frac{r}{K}a_2^2m_2(a_1-a_1')\equiv 0 \pmod{r}.
\]
Then $\frac{r}{K}a_2^2m_2(a_1-a_1')=rk$ for some $k\in\mathbb{Z}$ hence $K$ divides $a_2^2m_2(a_1-a_1')$. Since $a_2$ and $m_2$ are both relatively prime to $K$, we conclude that $a_1=a_1'$. 

For the other direction we will again make use of \cite[Proposition 2.14]{errs} a number of times. Assume $m_i\equiv n_i \pmod{K}$, $i=1,2$ and $\left(m_2^{-1}m_1-n_2^{-1}n_1\right)\frac{r(r-1)(r-2)}{3} \pmod{r}  \equiv 0 \pmod{r}$, then the entries, $y_t$, in the second row of the adjacency matrices are identical, and, it suffices to show for each $t$,
\[
\begin{aligned}
x_t'-x_t\equiv\left(m_2^{-1}m_1-n_2^{-1}n_1\right)\frac{r(r-1)(r-2)}{3K} +\sum_{\ell=1}^{r-2}\ell\frac{r}{K}(k_\ell'-k_\ell) \pmod{r}
\end{aligned}
\]
is an integer multiple of $\tfrac{r}{K}$. For the second term this is obvious. Additionally, it is obvious that this is also true for the first term, whenever $r$ is not a multiple of $3$. If $3\vert r$, then one finds that $3\vert m_2^{-1}m_1-n_2^{-1}n_1$, and the claim follows.

The adjacency matrices of each set of weights will then be identical after adding the third row to the first, and the first column to the $t$'th column in each an appropriate number of times.
\\
\\
(2). Proceeding as in part (3), assume that the quantum lens spaces coming from the weights $\underline{m}=(m_0,m_1,m_2,m_3)$ and $\underline{n}=(n_0,n_1,n_2,n_3)$ are isomorphic. We denote $x, a_1$ and $x',a_1'$ the elements coming from the adjacency matrix as written in Lemma \ref{lemmam1} corresponding to $\underline{m}$ and $\underline{n}$ respectively. See Notation \ref{notation:ai} for the definition of $a_1$. In a similar manner, it follows from Corollary \ref{lensclassification}, and a computation that
\[\frac{a_1'r}{K}\equiv \frac{a_1r}{K} \pmod{r},\]
It follows that $m_1\equiv n_1 \pmod{K}$. Moreover, we obtain from the adjacency matrices that
\vspace{0.2cm}
\\
\resizebox{\columnwidth}{!}{
$x'-x\equiv \left( n_2^{-1}n_1-m_2^{-1}m_1 \right)\frac{r(2r-K)(r-K)}{6K^2}+ \left( n_1^{-1}-m_1^{-1} \right)\frac{r(r-K)(K-1)}{4K} \pmod{r},$}
\vspace{0.2cm}
\\
and using \cite[Theorem 7.1]{errs}, we obtain $k_0,k_1\in \Zb$ such that
\[x'-x \equiv \frac{r(r+K)}{2K}k_0 +\frac{r}{K}k_1 \pmod{r}.\]
Consequently,
\vspace{0.2cm}
\\
\resizebox{\columnwidth}{!}{
$\left( n_2^{-1}n_1-m_2^{-1}m_1 \right)\frac{r(2r-K)(r-K)}{6K^2}+ \left( n_1^{-1}-m_1^{-1} \right)\frac{r(r-K)(K-1)}{4K}\equiv \frac{r(r+K)}{2K}k_0 +\frac{r}{K}k_1 \Mod{r}.$}
\vspace{0.1cm}
\\
Multiplying both sides by $2K$ yields
\[\left( n_2^{-1}n_1-m_2^{-1}m_1 \right)\frac{r(2r-K)(r-K)}{3K}+ \left( n_1^{-1}-m_1^{-1} \right)\frac{r(r-K)(K-1)}{2}\equiv 0 \Mod{r}.\]
Now, it is easy to check that the second term is always congruent to zero$\Mod{r}$, so we conclude that \[\left( n_2^{-1}n_1-m_2^{-1}m_1 \right)\frac{r(2r-K)(r-K)}{3K}\equiv 0 \pmod{r}.\]

Conversely, assuming that  
\begin{equation}\label{Eq:Main3}
    \left( n_2^{-1}n_1-m_2^{-1}m_1 \right)\frac{r(2r-K)(r-K)}{3K}=t \cdot r,
\end{equation}
we may argue as in the previous part, by showing that $x-x'$ is an integer multiple of $\frac{r}{K}$. It follows by a computation that
\[x-x'=\frac{r}{K}\left(\frac{2t+(m_1^{-1}-n_1^{-1})(r-K)(K-1)}{4} \right).\] Since $(m_1^{-1}-n_1^{-1})(r-K)(K-1)$ is necessarily divisible by 4 (to see this, one can consider the different parities of $K$ and $r$), it suffices to show that $t$ is even, which one can conclude by considering \eqref{Eq:Main3}

It remains to be shown that 
$\left( n_2^{-1}n_1-m_2^{-1}m_1 \right)\frac{r(2r-K)(r-K)}{3K} \equiv 0 \pmod{r}$
if and only if $\left( n_2^{-1}n_1-m_2^{-1}m_1 \right)\frac{r(r-1)(r-2)}{3} \equiv 0 \pmod{r}$. It is routine to show that if $3$ does not divide $r$, then $3$ divides either $2r-K$ or $r-K$, and the claim follows immediately. If $\left(n_2^{-1}n_1-m_2^{-1}m_1 \right)\frac{r(r-1)(r-2)}{3} \equiv 0 \pmod{r}$ and $3\vert r$, then $3\vert\left( n_2^{-1}n_1-m_2^{-1}m_1 \right)$ and one direction follows. The converse follows, remarking that $3\vert\left( n_2^{-1}n_1-m_2^{-1}m_1 \right)$ is always true if $\left( n_2^{-1}n_1-m_2^{-1}m_1 \right)\frac{r(2r-K)(r-K)}{3K} \equiv 0 \pmod{r}$.


The proof of (1) is identical to that of (2) remarking that the adjacency matrix corresponding to the system of weights $\underline{m}\coloneqq (m_0,m_1,m_2,m_3)$ with $\gcd(m_1,r)=K$ and $\gcd(m_i,r)=1$ if $i\neq 2$ is the anti-transpose of the adjacency matrix corresponding to the system $\underline{m}'\coloneqq(m_0,m_2,m_1,m_3)$. By \cite[Definition 1.7]{GS07}, the adjacency matrices are related by the identity
\[A_{\overline{L}_q^{7}(r;\underline{m})}= JA_{\overline{L}_q^{7}(r;\underline{m}')}^TJ,\]
where $J$ is the involutory matrix whose entries are $1$ on the second diagonal and $0$ elsewhere.
\end{proof}

\begin{remark}
In this paper we have only dealt with the case there a single weight is coprime to the order of the acting group, $r$. There is however no clear reason why a similar result should be unobtainable in a more general setting. In particular, if we consider a list of weights $(m_0,m_1,m_2,m_3)$, then the methods for computing the adjacency matrices of the corresponding graph, could very likely be identical or similar, albeit more tedious, to the ones used above if at least one of $m_1$ or $m_2$ is coprime to $r$. If both weights are coprime, it is likely that an entirely different approach to counting is necessary since all methods employed so far have required one of them to be a unit of $\Zb_r$.
\end{remark}

\end{document}